\documentclass{amsart}
\usepackage{amsmath,amssymb,amsthm,bbm,upgreek}
\DeclareMathAlphabet{\mathpzc}{OT1}{pzc}{m}{it}
\newcommand{\re}{\mathbb{R}}
\newcommand{\co}{\mathbb{C}}
\newcommand{\cc}{\mathcal{C}}
\newcommand{\aaa}{\mathcal{A}}
\newcommand{\ff}{\mathcal{F}}
\newcommand{\rr}{\mathcal{R}}
\newcommand{\qq}{\mathcal{Q}}
\newcommand{\pp}{\mathcal{P}}
\newcommand{\cl}{\mathpzc{c}}
\newcommand{\jl}{\mathpzc{j}}
\newcommand{\ddd}{\mathcal{D}}
\newcommand{\z}{\bar z}

\newcommand{\tz}{\tilde z}
\newcommand{\btz}{\bar{\tilde z}}

\newcommand{\tw}{\tilde w}
\newcommand{\rp}{\operatorname{Re}}
\newcommand{\ip}{\operatorname{Im}}
\newcommand{\ba}{{\mbox{\boldmath$\alpha$}}}
\newcommand{\sba}{{\mbox{\scriptsize\boldmath$\alpha$}}}
\newcommand{\bbeta}{{\mbox{\boldmath$\beta$}}}
\newcommand{\sbeta}{{\mbox{\scriptsize\boldmath$\beta$}}}
\newcommand{\bg}{{\mbox{\boldmath$\gamma$}}}
\newcommand{\boxx}{\rule{2.12mm}{3.43mm}}

\begin{document}

\title[CR singularities of real fourfolds in $\mathbb{C}^3$]{CR
singularities of real fourfolds in $\mathbb{C}^3$}

\author[A.\ Coffman]{Adam Coffman}

\address{Department of Mathematical Sciences \\ Indiana University -
Purdue University Fort Wayne \\ Fort Wayne, IN 46805-1499}

\email{CoffmanA@ipfw.edu}

\newtheorem{thm}{Theorem}[section]
\newtheorem{prop}[thm]{Proposition}
\newtheorem{lem}[thm]{Lemma}
\newtheorem{cor}[thm]{Corollary}
\newtheorem{conj}[thm]{Conjecture}

\theoremstyle{definition}

\newtheorem{defn}[thm]{Definition}
\newtheorem{notation}[thm]{Notation}
\newtheorem{example}[thm]{Example}

\theoremstyle{remark}
 
\newtheorem*{rem}{Remark}

\numberwithin{equation}{section}

\begin{abstract}
  CR singularities of real $4$-submanifolds in complex $3$-space are
  classified by using local holomorphic coordinate changes to
  transform the quadratic coefficients of the real analytic defining
  equation into a normal form.  The quadratic coefficients determine
  an intersection index, which appears in global enumerative formulas
  for CR singularities of compact submanifolds.
\end{abstract}



\maketitle

\section{Introduction}\label{sec0}
If a real $4$-manifold $M$ is embedded in $\co^3$, then for each point
$x$ on $M$ there are two possibilities: the tangent $4$-plane at $x$
is either a complex hyperplane in $\co^3$, so $M$ is said to be ``CR
singular'' at $x$, or it is not, so $M$ is said to be ``CR generic''
at $x$.  This article will consider the local extrinsic geometry of a
real analytically embedded $M$ near a CR singular point, by finding
invariants under biholomorphic coordinate changes.  The main result is
a classification of quadratic normal forms for the defining function.
The matrix algebra leading to the classification is worked out in
Section \ref{sec1}, and then summarized in Section \ref{sum} after the
geometric interpretation is developed in Sections \ref{tc} --
\ref{complexification}.

The analysis of normal forms near CR singular points is part of the
program of studying the local equivalence problem for real
$m$-submanifolds of $\co^n$.  In this paper we consider the $m=2n-2$
case (``codimension $2$''), focusing on real $4$-manifolds in $\co^3$,
since the $m=n=2$ case is well-known and larger dimensions seem to
lead to difficult computations.

In Section \ref{tc}, we recall some of Lai's formulas relating the
global topology of a real submanifold to the number of its CR singular
points, counted with sign according to an intersection number.  In
Section \ref{atc}, we derive a simple expression that calculates the
intersection index in terms of the coefficients in the local defining
equation, generalizing the well-known $m=n=2$ case, where CR singular
points of compact surfaces in $\co^2$ can be counted according to
their elliptic or hyperbolic nature as determined by the local Bishop
invariant.  Section \ref{vge} gives some concrete examples of compact
real $4$-manifolds immersed in $\co^3$ or $\co P^3$ to illustrate the
enumerative formulas and local invariants.

\section{Normal forms for CR singularities}\label{sec1}

Let $n\ge2$ and $m=2n-2$, so a real $m$-submanifold $M$ of a complex
$n$-manifold has real codimension $2$.  Considering $n$ in general
shows how the well-known $(m,n)=(2,2)$ case is related to the
higher-dimensional cases, including $(4,3)$.  In this Section, we are
only interested in a small coordinate neighborhood, so we will let the
ambient complex space be $\co^n$, with coordinates $z_1,\ldots,z_n$.
We will also use the abbreviations $z=(z_1,\ldots,z_{n-1})^T$ and
$\vec z=(z_1,\ldots, z_n)^T$ (both column vectors).  The real and
imaginary parts of the coordinate functions are labeled
$z_k=x_k+iy_k$.

\subsection{Local defining equations and transformations}\label{ldet}

\ 

We begin by assuming $M$ is a real analytic $(2n-2)$-submanifold in
$\co^n$ with a CR singularity at some point: the tangent space at that
point is a complex hyperplane.  We are interested in the invariants of
$M$ under local biholomorphic transformations.  By a translation that
moves the CR singular point to the origin $\vec0$, and then a complex
linear transformation of $\co^n$, the tangent $(2n-2)$-plane $T_{\vec0}M$
can be assumed to be the $(z_1,\ldots,z_{n-1})$-subspace.  Then there
is some neighborhood $\Delta$ of the origin in $\co^n$ so that the
defining equation of $M$ in $\Delta$ is in the form of a graph over a
neighborhood $\ddd$ of the origin in the complex subspace $T_{\vec0}M$:
\begin{eqnarray*}
  z_n&=&h(z_1,\z_1,\ldots,z_{n-1},\z_{n-1})=h(z,\z)
\end{eqnarray*}
where $h(z,\z)$ is a complex valued real analytic function defined for
$z\in\ddd\subseteq T_{\vec0}M$, and vanishing to second order at
$z=(0,\ldots,0)^T$.  Once $M$ is in this ``standard position,'' the
complex defining function $h(z,\z)$ is of the following form:
\begin{eqnarray}
  h(z,\z)&=&z^TQz+\z^TRz+\z^TS\z+e(z,\z),\label{eq0}
\end{eqnarray}
where $Q$, $R$, $S$ are complex $(n-1)\times(n-1)$ coefficient
matrices, $z^T$ and $\z^T$ are row vectors, and $e(z,\z)$ is a real analytic
function on $\ddd$ vanishing to third order at $z=(0,\ldots,0)^T$.
The matrices $Q$ and $S$ can be assumed to be complex symmetric.  It
can also be assumed that $\ddd$ is small enough so that the function
$h(z,\z)$ can be expressed as the restriction to
$\{(z,w)\in\ddd\times\ddd:w_k=\z_k\}$ of the multi-indexed series:
\begin{equation}\label{eq-2}
  h(z,w)=z^TQz+w^TRz+w^TSw+\sum_{|\sba|+|\sbeta|\ge3}e^{\sba\sbeta}z^\sba
  w^\sbeta,
\end{equation}
which converges on some set
\begin{equation}\label{eq5}
  \ddd_{\cl}=\{(z,w):|z_k|<\epsilon,|w_k|<\epsilon\}\subseteq\co^{2n-2}
\end{equation}
to a complex analytic function.

\begin{defn}\label{not33}
  A (formal, with multi-indexed complex coefficient $e^{\sba\sbeta}$)
  monomial of the form $e^{\sba\sbeta}z^\sba
  w^\sbeta=e^{\sba\sbeta}z_1^{\alpha_1}\cdots
  z_{n-1}^{\alpha_{n-1}}w_1^{\beta_1}\cdots w_{n-1}^{\beta_{n-1}}$ has
  ``degree''
  $|\ba|+|\bbeta|=\alpha_1+\cdots\alpha_{n-1}+\beta_1+\cdots\beta_{n-1}$.
  A power series (convergent or formal) $e(z,w)=\sum
  e^{\sba\sbeta}z^\sba w^\sbeta$ with $e^{\sba\sbeta}=0$ for all terms
  of degree $|\ba|+|\bbeta|<{\bf d}$ will be abbreviated $e(z,w)=O({\bf d})$.
\end{defn}
The above notation applies to expressions of the form $e(z,\z)$, for
example, the last term in Equation (\ref{eq0}) is $e(z,\z)=O(3)$.

We consider the effect of a coordinate change of the following form:
\begin{eqnarray}
  \tz_1&=&z_1+p_1(z_1,\ldots,z_n)\label{eq1}\\
  &\vdots&\nonumber\\
  \tz_n&=&z_n+p_n(z_1,\ldots,z_n),\nonumber
\end{eqnarray}
abbreviated $\vec\tz=\vec z+\vec p(\vec z)$, where $p_1(\vec
z),\ldots,p_n(\vec z)$ are holomorphic functions defined by series
centered at $\vec0$ with no linear or constant terms.  Since this
transformation of $\co^n$ has its linear part equal to the identity
map, it is invertible on some neighborhood of the origin and preserves
the form of (\ref{eq0}).  In the following calculations, we will
neglect considering the size of that neighborhood, and consider only
points close enough to the origin.

As the first special case of a transformation of the form (\ref{eq1})
to be used, let $p_1(\vec z),\ldots,p_{n-1}(\vec z)$ be identically
zero, so $\tz=z$, and let $p_n(\vec z)$ be a homogeneous quadratic
polynomial in $z_1,\ldots,z_{n-1}$, so $p_n(\vec z)=z^TQ^\prime z$ for
some complex symmetric $(n-1)\times(n-1)$ matrix $Q^\prime$.  Given a
point on $M$ near $\vec0$, its coordinates $\vec z=(z_1,\ldots,z_n)^T$
satisfy $z_n-h(z,\z)=0$.  The new coordinates at that point satisfy:
\begin{eqnarray}
  \tz_n&=&z_n+p_n(\vec
  z)\label{eq4}\\ &=&z^TQz+\z^TRz+\z^TS\z+e(z,\z)+p_n(\vec
  z)\nonumber\\ &=&\tz^T(Q+Q^\prime)\tz+\btz^TR\tz+\btz^TS\btz+e(\tz,\btz).\nonumber
\end{eqnarray}
So, such a quadratic transformation changes the coefficient matrix
$Q$, but all the other coefficients of the new equation, $\tz_n-\tilde
h(\tz,\btz)=0$, are the same.  Choosing $Q^\prime=\bar S-Q$ (and
dropping the tilde notation), the defining equation in the new
coordinates is:
\begin{eqnarray}
  z_n&=&z^T\bar Sz+\z^TRz+\z^TS\z+e(z,\z),\label{eq6}
\end{eqnarray}
so the first and third terms have a real valued sum and $e(z,\z)$ is still
$O(3)$.

Next, we consider some linear transformations of $\co^n$, but only
those which fix, as a set, the complex tangent hyperplane
$T_{\vec0}M=\{z_n=0\}$, so they are of the form
\begin{equation}
  \vec\tz_{n\times1}=C_{n\times n}\vec
z_{n\times1}=\left(\begin{array}{cccc}c_{1,1}&\ldots&c_{1,n-1}&c_{1,n}\\\vdots&&\vdots&\vdots\\c_{n-1,1}&\ldots&c_{n-1,n-1}&c_{n-1,n}\\0&\ldots&0&c_{n,n}\end{array}\right)\vec
z_{n\times1},\label{eq9}
\end{equation}
with complex entries and nonzero determinant, so $c_{n,n}\ne0$.  The
inverse matrix has the block form
$$C^{-1}=\left(\begin{array}{cccc}&&&*\\&A_{(n-1)\times(n-1)}&&\vdots\\&&&*\\0&\ldots&0&c_{n,n}^{-1}\end{array}\right),$$
where the block $A$ in $C^{-1}$ does not depend on the entries
$c_{1,n},\ldots,c_{n,n}$.  In the special case where
$c_{1,n}=\ldots=c_{n-1,n}=0$, $C$ has a block diagonal pattern and so
does its inverse, so $z=A\tz$.  In the coordinate system defined by
such a linear transformation, the new defining equation is
\begin{eqnarray}
  \tz_n&=&c_{n,n}z_n\nonumber\\ &=&c_{n,n}\cdot(z^T\bar
  Sz+\z^TRz+\z^TS\z+e(z,\z))\nonumber\\ &=&c_{n,n}\cdot(\tz^TA^T\bar
  SA\tz+\btz^T\bar A^TRA\tz+\btz^T\bar A^TS\bar A\btz+\tilde
  e(\tz,\btz))\label{eq7},
\end{eqnarray}
where the new higher order part is still real analytic but may have a
different domain of convergence.  

If the coefficients $c_{1,n},\ldots,c_{n-1,n}$ were nonzero, they
would contribute only terms of degree at least $3$, not affecting the
quadratic terms in (\ref{eq7}).  Similarly, allowing a coordinate
change with nonlinear terms, as in (\ref{eq1}), would only introduce
terms of degree at least $3$, or, as in (\ref{eq4}), arbitrarily alter
the first quadratic term.  So, under a general transformation,
\begin{equation}\label{eq-1}
  \vec\tz=C\vec z+\vec p(\vec z)
\end{equation}
(which combines (\ref{eq1}) and (\ref{eq9}), and preserves the
standard position, (\ref{eq0})), the only interesting effect on the
quadratic part of $h(z,\z)$ is that the coefficient matrices are transformed
as:
\begin{equation}\label{eq3}
  (R,S)\mapsto(c_{n,n}\bar A^TRA,c_{n,n}\bar A^TS\bar A).
\end{equation}
The first invariant to notice is the pair
$(\mbox{rank}(R),\mbox{rank}(S))$.  The rank of the concatenated
matrix $(R|S)_{ (n-1)\times(2n-2)}$ is also an invariant under this
action.

The group of invertible matrices $A$ has $(n-1)^2$ complex dimensions,
and the group of scalars $c_{n,n}$ is one-dimensional; however, if $A$
is a real multiple $\lambda$ of the identity matrix $\mathbbm 1$, then
its action can be canceled by choosing $c_{n,n}=\lambda^{-2}$.  So,
there are at most $2((n-1)^2+1)-1=2n^2-4n+3$ real parameters in the
group action.  The coefficient matrices $R$ and $S$ have $(n-1)^2$ and
$(n-1)n/2$ complex dimensions, for a total of $3n^2-5n+2$ real
dimensions.  The number of coefficients always exceeds the number of
parameters in the group action, so we expect infinitely many
equivalence classes of matrix pairs, distinguished by continuous
invariants.

\subsection{Degrees of flatness}\label{df}

We continue with the assumption that $M$ is a real analytic
submanifold of $\co^n$ with real codimension $2$.

\begin{defn}\label{def2.2}
  A manifold $M$ in standard position (\ref{eq0}) has a defining
  function $h(z,\z)$ in a ``quadratically flat normal form'' if the
  quadratic part, $z^TQz+\z^TRz+\z^TS\z$, of its defining function is
  a real valued polynomial.  A manifold $M\subseteq\co^n$ with a CR
  singular point $\vec x\in M$ is ``quadratically flat'' at $\vec x$
  if, after $M$ is put into standard position (\ref{eq0}) by a complex
  affine transformation $\vec\tz=L_{n\times n}\cdot(\vec z-\vec x)$,
  there is a local holomorphic coordinate change (\ref{eq-1}) such
  that in the new coordinates, $M$ has a defining function in a
  quadratically flat normal form.
\end{defn}
The definition of ``quadratically flat normal form'' is equivalent to
$Q=\bar S$ and $R=\bar R^T$ (so $R$ is Hermitian symmetric).
Considering (\ref{eq6}) and the transformation rule (\ref{eq3}), for
$M$ in standard position, the ``quadratically flat'' property is
equivalent to $R$ being a complex scalar multiple of a Hermitian
symmetric matrix.

The notion of quadratic flatness is the ${\bf d}=2$ special case
of the following generalization to higher degree.
\begin{defn}\label{def2.3}
  For ${\bf d}\ge2$, a real analytic manifold $M$ in standard position
  (\ref{eq0}) has a defining function in a ``${\bf d}$-flat normal form'' if
  the defining equation in a neighborhood of $\vec0$
  is $$z_n=h(z,\z)={\bf r}(z,\z)+O({\bf d}+1)$$ for some real valued polynomial
  ${\bf r}(z,\z)$.  A manifold $M\subseteq\co^n$ with a CR singular point
  $\vec x\in M$ is ``${\bf d}$-flat'' at $\vec x$ if, after $M$ is put into
  standard position (\ref{eq0}) by a complex affine transformation
  $\vec\tz=L_{n\times n}\cdot(\vec z-\vec x)$, there is a local
  holomorphic coordinate change (\ref{eq-1}) such that in the new
  coordinates, $M$ has a defining function in a ${\bf d}$-flat normal form.
\end{defn}
\begin{defn}\label{def2.4}
  A real analytic $(2n-2)$-manifold $M\subseteq\co^n$ is ``formally
  flattenable'' at a CR singular point $\vec x\in M$ if it is ${\bf
  d}$-flat at $\vec x$ for every ${\bf d}\ge2$.
\end{defn}
\begin{defn}\label{def2.5}
  A manifold $M\subseteq\co^n$ with a CR singular point $\vec x\in M$
  is ``holomorphically flat'' at $\vec x$ if, after $M$ is put into
  standard position (\ref{eq0}) by a complex affine transformation
  $\vec\tz=L_{n\times n}\cdot(\vec z-\vec x)$, there is a local
  holomorphic coordinate change (\ref{eq-1}) such that in the new
  coordinates, the defining function (\ref{eq0}) is real valued.
\end{defn}
  By the Definition, if $M$ is holomorphically flat near a CR singular
  point, then there is a local coordinate system around the point so
  that a neighborhood of $\vec0$ in $M$ is contained in the real
  hyperplane $\ip(z_n)=0$.  By the well-known normal form result of
  \'E.\ Cartan that a real analytic nonsingular Levi flat hypersurface
  is locally biholomorphically equivalent to a real hyperplane, the
  local notion of $M$ being holomorphically flat at $\vec x$ is
  equivalent to the (more coordinate-free) property that there exists
  a real analytic nonsingular Levi flat hypersurface containing a
  neighborhood of $\vec x$ in $M$.

\subsection{The $m=n=2$ case}\label{mn2}

\ 

When $m=n=2$, $M$ is a real surface in $\co^2$ with a CR singular
point.  For $M$ in standard position (\ref{eq0}), the coefficient
matrices are size $1\times1$, and can be written as complex constant
coefficients.  The action of (\ref{eq3}) becomes
$(R,S)\mapsto(c_{2,2}|\alpha|^2R,c_{2,2}\bar\alpha^2S)$ for nonzero
complex constants $c_{2,2}$ and $\alpha$, where
$A_{1\times1}=(\alpha)$.  If $R\ne0$, then $(R,S)$ can then be
transformed into $(1,\gamma_1)$, $\gamma_1\ge0$.  If $R=0$, then there
are two normal forms: $(0,1)$ and $(0,0)$.  The quadratic normal forms
for the defining function of $M$ are then:
\begin{eqnarray}
  z_2&=&z_1\z_1+\gamma_1\cdot(z_1^2+\z_1^2)+e(z_1,\z_1),\ \gamma_1\ge0\mbox{, or}\label{eq34}\\
  z_2&=&z_1^2+\z_1^2+e(z_1,\z_1)\mbox{, or}\nonumber\\
  z_2&=&e(z_1,\z_1).\nonumber
\end{eqnarray}  
So, $\gamma_1$ is the well-known Bishop invariant (\cite{bishop}) and
the second case is $\gamma_1=+\infty$.  This calculation of the
quadratic normal forms shows that any surface $M$ with a CR singular
point is quadratically flat (Definition \ref{def2.2}) at that point.

The normalization of the cubic terms depends on $\gamma_1$; all the
cases $0\le\gamma_1\le\infty$ are surveyed in \cite{mem} \S 5, and we
recall a few examples here.

For $\gamma_1\in(0,\frac12)\cup(\frac12,1)\cup(1,\infty)$, it was shown
by \cite{mw} that the cubic terms of $e(z,\z)$ can be eliminated by a
holomorphic coordinate change near the origin, and so $M$ is $3$-flat.

Any $M$ with $\gamma_1=\frac12$ is $3$-flat; although there may be some
cubic terms that cannot be eliminated by a holomorphic coordinate
change, such terms can always be made real valued.  For $\gamma_1=1$,
there are some $M$ which are not $3$-flat.

It was also shown by \cite{mw}, and \cite{hk} respectively, that for
$0<\gamma_1<\frac12$, and $\gamma_1=0$, $M$ is holomorphically flat.  It
was proved by \cite{gong} that there exists some $M$ with
$\gamma_1>\frac12$ which is formally flattenable but not holomorphically
flat.

\subsection{The $m=4$, $n=3$ case}\label{mn43}

\subsubsection{A quadratic normal form}\label{qnf}

In this, the main case of this paper, $M$ is a $4$-manifold in
$\co^3$, which we assume is given in the form (\ref{eq6}).  The
quadratic coefficient matrices $R$, $S$ are size $2\times2$, so there
are $7$ independent complex coefficients, and, as previously
calculated, $9$ real parameters in the group action (\ref{eq3}).  One
expects that in general, attempting to put the pair $(R,S)$ into a
normal form will leave $5$ continuous real invariants.

We choose to begin the normalization by considering the action of the
transformation $R_{2\times2}\mapsto c\bar A^TRA$, where $c=c_{3,3}$ is
a nonzero scalar and $A$ is an invertible $2\times2$ complex matrix.
Conveniently, the problem of finding representative matrices for the
orbits of this action has already been solved in \cite{ac}:

\begin{prop}[\cite{ac} Theorem 4.3]\label{thm4.3}
  Given a complex $2\times2$ matrix $R$, there is exactly one of the
  following normal forms $N$ such that $N=c\bar A^TRA$ for some
  nonzero complex $c$ and some invertible complex $A_{2\times2}$:
  \begin{enumerate} \item
  $\left(\begin{array}{cc}1&0\\0&e^{i\theta}\end{array}\right)$,
    $0\le\theta\le\pi$; \item
    $\left(\begin{array}{cc}0&1\\\tau&0\end{array}\right)$,
      $0\le\tau<1$; \item
      $\left(\begin{array}{cc}0&1\\1&i\end{array}\right)$; \item
        $\left(\begin{array}{cc}1&0\\0&0\end{array}\right)$; \item
          $\left(\begin{array}{cc}0&0\\0&0\end{array}\right)$.
  \end{enumerate}  \boxx
\end{prop}
For most values of the invariants $\theta$, $\tau$, these normal forms
are not Hermitian symmetric --- unless $R$ is already a complex
multiple of a Hermitian matrix, one would not expect $c\bar A^TRA$ to
be Hermitian.  So, unlike the $m=n=2$ case from Subsection \ref{mn2},
the quadratic part of the defining function $h(z,\z)$ generally cannot
be made real valued by a holomorphic coordinate change.  For a
manifold in standard position (\ref{eq0}), the following are
equivalent: (I) $M$ is quadratically flat at $\vec0$; (II) $R$ is a
multiple of a Hermitian matrix; (III) $N=c\bar A^TRA$, where $N$ is
one of the following normal forms from the Proposition: Case (1) with
$\theta=0$ or $\pi$, Case (4), or Case (5).

Using $N$ from the above Proposition and introducing the notation
$P=2\bar S_{2\times2}$, (\ref{eq6}) becomes:
\begin{eqnarray}
  \ \ \ \ \ \ \ \ z_3&=&(\z_1,\z_2)N\left(\begin{array}{c}z_1\\z_2\end{array}\right)+\rp\left((z_1,z_2)P\left(\begin{array}{c}z_1\\z_2\end{array}\right)\right)+e(z_1,\z_1,z_2,\z_2).\label{eq12}
\end{eqnarray}

The linear action of $C$ (\ref{eq9}, \ref{eq7}), followed by another
holomorphic transformation (\ref{eq4}), preserves the form of the
defining equation (\ref{eq12}), and acts on the coefficient matrices
by the transformation:
\begin{equation}\label{eq13}
  (N,P)\mapsto(c\bar A^TNA,\bar cA^TPA).
\end{equation}
Since $N$ is already normalized, to find a normal form for
(\ref{eq12}), we consider only pairs $(c,A)$ that preserve $N$:
$N=c\bar A^TNA$.  Since this depends on the nature of the various
matrices $N$ appearing in the Proposition, we proceed in cases.

In each case, let
$P=\left(\begin{array}{cc}a&b\\b&d\end{array}\right)$, with complex
entries $a$, $b$, $d$, and let
$A=\left(\begin{array}{cc}\alpha&\beta\\\gamma&\delta\end{array}\right)$,
with complex entries and nonzero determinant.

Case (1a).  For $N=
\left(\begin{array}{cc}1&0\\0&e^{i\theta}\end{array}\right)$, with
$0<\theta<\pi$, if $N=c\bar A^TNA$, then $c$ is real (this follows
from calculating determinants, for example), and 
$$N=c\bar
A^TNA=c\left(\begin{array}{cc}\bar\alpha&\bar\gamma\\\bar\beta&\bar\delta\end{array}\right)\left(\begin{array}{cc}1&0\\0&e^{i\theta}\end{array}\right)\left(\begin{array}{cc}\alpha&\beta\\\gamma&\delta\end{array}\right)$$
implies 
\begin{eqnarray}
1&=&c\cdot(\alpha\bar\alpha+\gamma\bar\gamma e^{i\theta}),\label{eq14}\\
0&=&\alpha\bar\beta+\gamma\bar\delta e^{i\theta},\nonumber\\
e^{i\theta}&=&c\cdot(\beta\bar\beta+\delta\bar\delta e^{i\theta})\nonumber.
\end{eqnarray}
It follows that $\gamma=\beta=0$, $c=|\alpha|^{-2}$, and
$|\delta|=|\alpha|$.  So the action of (\ref{eq13}) is that $P$ can be
transformed to: $$\bar
cA^TPA=\frac1{|\alpha|^2}\left(\begin{array}{cc}\alpha&0\\0&\delta\end{array}\right)\left(\begin{array}{cc}a&b\\b&d\end{array}\right)\left(\begin{array}{cc}\alpha&0\\0&\delta\end{array}\right)=\left(\begin{array}{cc}a\frac{\alpha^2}{|\alpha|^2}&b\frac{\alpha\delta}{|\alpha\delta|}\\b\frac{\alpha\delta}{|\alpha\delta|}&d\frac{\delta^2}{|\delta|^2}\end{array}\right).$$
When $a$ and $d$ are both nonzero, $\alpha$ and $\delta$ can be
chosen to rotate them onto the positive real axis.  The value of $b$
cannot be normalized any further except that $P$ with positive $a$,
$d$, and complex $b$ is equivalent to the matrix with the same $a$,
$d$, but opposite value for $b$.

If $a=0$ or $d=0$, then $\alpha$ and $\delta$ can be chosen to
transform $P$ into a matrix with all non-negative entries.

Case (1b).  For $N=\left(\begin{array}{cc}1&0\\0&1\end{array}\right)$,
  the $\theta=0$ case of Proposition \ref{thm4.3}, $c$ is real as in
  the previous case, and by the calculation analogous to (\ref{eq14})
  with $\theta=0$, in fact $c$ is positive, so
  $c=\frac{+1}{|\det(A)|}$.  The equation $N=c\bar
  A^TNA=\overline{(\sqrt{c}A)}^TN(\sqrt{c}A)$ shows $A$ is a real
  multiple of a unitary matrix, and conversely if $A=rU$ for some real
  $r$ and unitary $U$, then $(c,A)=(r^{-2},rU)$ stabilizes $N$.  So
  the action of (\ref{eq13}) is that $P$ can be transformed to: $$\bar
  cA^TPA=\frac1{|\det(A)|}A^TPA=r^{-2}(rU)^TP(rU)=U^TPU.$$ The normal
  form problem for $P$ is thus reduced to finding a normal form for a
  complex symmetric $2\times2$ matrix under the relation of congruence
  by a unitary matrix.  This problem has a well-known solution by
  Takagi (\cite{hj1} \S 4.4, see also Theorem 5 of \cite{hua}), which
  says that a complex symmetric matrix has a diagonal normal form
  under unitary congruence, with non-negative real entries.  These
  entries can be re-ordered by a unitary transformation, so a normal
  form is $\left(\begin{array}{cc}a&0\\0&d\end{array}\right)$ with
    $0\le a\le d$.

Case (1c).  For
$N=\left(\begin{array}{cc}1&0\\0&-1\end{array}\right)$, the
$\theta=\pi$ case of Proposition \ref{thm4.3}, $c$ is real.  Briefly
neglecting $c$, we consider the group of invertible matrices $A$
preserving $N$.  The condition $\bar A^TNA=N$ is equivalent to $AN\bar
A^T=N$.  For any symmetric coefficient matrix $P$, define an
auxiliary matrix $B=\bar P^TNP$, which is Hermitian symmetric.  Then
the congruence action $P^\prime=A^TPA$ transforms the product:
$$B^\prime=\overline{P^\prime}^TNP^\prime=\overline{(A^TPA)}^TN(A^TPA)=\bar
A^T\bar P^T\bar ANA^TPA=\bar A^T\bar P^TNPA,$$ so $B^\prime$ is
related to $B$ by Hermitian congruence.  So, the action (\ref{eq13})
of $A$ on $(N,P)$ (with $c=1$) has been temporarily replaced by the
action of simultaneous Hermitian congruence on the pair $(N,B)$ of
Hermitian symmetric matrices.  This normal form problem is considered
by \cite{hj1}, and more recently by \cite{hs} and \cite{lr}.
Recalling that $N$ is Hermitian congruent to the matrix
$N^\prime=\left(\begin{array}{cc}0&1\\1&0\end{array}\right)$ (the
$\tau\to1^-$ limit of Case (2) from the Proposition), the result from
\cite{hj1}, \cite{hs}, \cite{lr} is that for any Hermitian $B$, the
pair $(N,B)$ is equivalent under simultaneous Hermitian congruence to
exactly one pair from the following list:
\begin{itemize}
  \item $\left(\left(\begin{array}{cc}1&0\\0&-1\end{array}\right),\left(\begin{array}{cc}k_1&0\\0&k_2\end{array}\right)\right)$, $k_1,k_2\in\re$;
  \item $\left(\left(\begin{array}{cc}0&1\\1&0\end{array}\right),\left(\begin{array}{cc}0&k\\k&1\end{array}\right)\right)$, $k\in\re$;
  \item $\left(\left(\begin{array}{cc}0&1\\1&0\end{array}\right),\left(\begin{array}{cc}0&x+iy\\x-iy&0\end{array}\right)\right)$, $x\in\re$, $y>0$.
\end{itemize}
We continue with Case (1c) by splitting into subcases corresponding to
the above three intermediate normal forms.

Case (1ci).  There is some nonsingular matrix $A$ so that the
Hermitian pair $(N,\bar P^TNP)$ is simultaneously diagonalized and
$P^\prime=A^TPA$ is a complex symmetric matrix
$\left(\begin{array}{cc}a&b\\b&d\end{array}\right)$ (using the same
  place-holding letters as in Case (1a) even though $P$ has been
  transformed once already) satisfying
\begin{equation}\label{eq17}
B^\prime=\overline{P^\prime}^TNP^\prime=\overline{\left(\begin{array}{cc}a&b\\b&d\end{array}\right)}^T\left(\begin{array}{cc}1&0\\0&-1\end{array}\right)\left(\begin{array}{cc}a&b\\b&d\end{array}\right)=\left(\begin{array}{cc}k_1&0\\0&k_2\end{array}\right).
\end{equation}
By another transformation, of the form $e^{i\xi}{\mathbbm 1}$ (which
does not affect the pair $(N,B^\prime)$), we may assume that the entry
$b$ of $P^\prime$ satisfies $b\ge0$.  It then follows from expanding
(\ref{eq17}) that the entries of $P^\prime$ must satisfy either $b=0$
or $d=\bar a$.  In the $b=0$ case, a transformation of the form
$\left(\begin{array}{cc}\alpha&0\\0&\delta\end{array}\right)$, with
  $|\alpha|=|\delta|=1$, preserves $N$ and puts $P^\prime$ into a
  diagonal normal form with non-negative real entries.  By a
  transformation of the form
  $(c,A)=\left(-1,\left(\begin{array}{cc}0&i\\i&0\end{array}\right)\right)$,
    these entries can be interchanged, so a unique normal form is
    $\left(\begin{array}{cc}a&0\\0&d\end{array}\right)$ with $0\le
      a\le d$.

In the $d=\bar a$ case, the same type of diagonal transformation with
$\delta=1/\alpha$ puts $P^\prime$ into the form
$\left(\begin{array}{cc}a&b\\b&a\end{array}\right)$ with $b>0$ and
  $a\ge0$.  However, this simplifies even further, and is not always
  different from the previous ($b=0$) case.  When $0<b<a$, a
  transformation of the form
  $$(c,A)=\left(\frac{-1}{2\sqrt{a^2-b^2}(-a+\sqrt{a^2-b^2})},\left(\begin{array}{cc}b&\!\!-a+\sqrt{a^2-b^2}\\-a+\sqrt{a^2-b^2}\!\!&b\end{array}\right)\right)$$
    preserves $N$ and diagonalizes $P^\prime$ to
    $\sqrt{a^2-b^2}\cdot{\mathbbm 1}$.  When $0\le a<b$, $N$ and
    $P^\prime$ cannot be simultaneously diagonalized, but a
    transformation of the
    form $$(c,A)=\left(\frac{1}{4b\sqrt{b^2-a^2}},\left(\begin{array}{cc}-a+i(b+\sqrt{b^2-a^2})&b-\sqrt{b^2-a^2}-ia\\b-\sqrt{b^2-a^2}-ia&-a+i(b+\sqrt{b^2-x^2})\end{array}\right)\right)$$
    preserves $N$ and takes $P^\prime$ to
    $\frac{-i}b\sqrt{b^2-a^2}(a-i\sqrt{b^2-a^2})N^\prime$, which can
    be rotated by $A=e^{i\theta}{\mathbbm 1}$ to
    $\sqrt{b^2-a^2}N^\prime$.  In the $0<a=b$ case, $P^\prime$ has
    rank $1$ but $(N,P^\prime)$ is not simultaneously diagonalizable,
    so it is inequivalent to the $b=0$ case with rank $1$, where
    $0=a<d$.  A transformation of the
    form $$(c,A)=\left(\frac1{4a},\left(\begin{array}{cc}1+a&1-a\\1-a&1+a\end{array}\right)\right)$$
      preserves $N$ and normalizes $a$ to $1$.

Case (1cii).  If there is no transformation simultaneously
diagonalizing $(N,B)$, then there is some nonsingular matrix $A$ so
that $\bar A^TNA=N^\prime$, $A^TPA=P^\prime$, and $B^\prime=\bar
A^TBA$ equals either the second or third normal form from the above
list --- in this subcase we consider the second.  The property $\bar
A^TNA=N^\prime$ is equivalent to $\bar AN^\prime A^T=N$,
so $$B^\prime=\bar A^TBA=\bar A^T\bar P^TNPA=\bar A^T\bar P^T\bar
AN^\prime A^TPA=\overline{(A^TPA)}^TN^\prime(A^TPA).$$ With notation
as in the previous case, $P^\prime=A^TPA$ is a complex symmetric
matrix $\left(\begin{array}{cc}a&b\\b&d\end{array}\right)$ satisfying
\begin{equation}\label{eq18}
B^\prime=\overline{P^\prime}^TN^\prime
P^\prime=\overline{\left(\begin{array}{cc}a&b\\b&d\end{array}\right)}^T\left(\begin{array}{cc}0&1\\1&0\end{array}\right)\left(\begin{array}{cc}a&b\\b&d\end{array}\right)=\left(\begin{array}{cc}0&k\\k&1\end{array}\right).
\end{equation}
By another transformation, of the form $e^{i\xi}{\mathbbm 1}$ (which
does not affect the pair $(N^\prime,B^\prime)$), we may assume that
the entry $a$ of $P^\prime$ satisfies $a\ge0$.  However, it follows
from expanding the product in (\ref{eq18}) that there are no solutions
of (\ref{eq18}) with $a>0$, so $a=0$.  Then (\ref{eq18}) becomes
\begin{equation*}
\overline{\left(\begin{array}{cc}0&b\\b&d\end{array}\right)}^T\left(\begin{array}{cc}0&1\\1&0\end{array}\right)\left(\begin{array}{cc}0&b\\b&d\end{array}\right)=\left(\begin{array}{cc}0&b\bar b\\b\bar b&b\bar d+\bar bd\end{array}\right)=\left(\begin{array}{cc}0&k\\k&1\end{array}\right),
\end{equation*}
so neither $b$ nor $d$ is $0$.  By yet another scalar transformation,
we may assume that $b>0$, without changing the RHS of the above
equation, so $b$ is an invariant determined by $k$, and the equality
of entries $b\cdot(d+\bar d)=1$ implies $\rp(d)\ne0$.  A
transformation
$(c,A)=\left(\frac1r,\left(\begin{array}{cc}1&is\\0&r\end{array}\right)\right)$,
  with $r$ and $s$ real, preserves $N^\prime$ and transforms
  $P^\prime=\left(\begin{array}{cc}0&b\\b&d\end{array}\right)$ into
    $\left(\begin{array}{cc}0&b\\b&2bis+rd\end{array}\right)$.  Since
      $\rp(d)\ne0$ and $b>0$, $r$ and $s$ can be chosen to normalize
      $d$ to $1$.

Case (1ciii).  The third case starts with the same steps as the previous case, with $\bar A^TNA=N^\prime$ and $A^TPA=P^\prime=\left(\begin{array}{cc}a&b\\b&d\end{array}\right)$ satisfying
\begin{equation}\label{eq19}
B^\prime=\overline{P^\prime}^TN^\prime
P^\prime=\overline{\left(\begin{array}{cc}a&b\\b&d\end{array}\right)}^T\left(\begin{array}{cc}0&1\\1&0\end{array}\right)\left(\begin{array}{cc}a&b\\b&d\end{array}\right)=\left(\begin{array}{cc}0&x+iy\\x-iy&0\end{array}\right),
\end{equation}
with $y>0$.  By another transformation, of the form
$e^{i\xi}{\mathbbm 1}$ (which does not affect the pair
$(N^\prime,B^\prime)$), we may assume that the entry $a$ of $P^\prime$
satisfies $a\ge0$.  Then it follows from expanding the product in
(\ref{eq19}) that $a>0$, $b=0$, and $ad=x+iy$, an invariant quantity.
A transformation
$(c,A)=\left(a,\left(\begin{array}{cc}1/a&0\\0&1\end{array}\right)\right)$
preserves $N^\prime$ and transforms
$P^\prime=\left(\begin{array}{cc}a&0\\0&d\end{array}\right)$ into
$\left(\begin{array}{cc}1&0\\0&ad\end{array}\right)=\left(\begin{array}{cc}1&0\\0&x+iy\end{array}\right)$.

Case (2a).  For $N=\left(\begin{array}{cc}0&1\\
\tau&0\end{array}\right)$, with $0<\tau<1$, if $N=c\bar A^TNA$,
then $c$ is real, and
$$N=c\bar
A^TNA=c\left(\begin{array}{cc}\bar\alpha&\bar\gamma\\\bar\beta&\bar\delta\end{array}\right)\left(\begin{array}{cc}0&1\\
\tau&0\end{array}\right)\left(\begin{array}{cc}\alpha&\beta\\\gamma&\delta\end{array}\right)$$
implies
\begin{eqnarray}
0&=&\alpha\bar\gamma\tau+\gamma\bar\alpha,\label{eq16}\\
0&=&\beta\bar\delta\tau+\bar\beta\delta,\nonumber\\
1&=&c\cdot(\beta\bar\gamma\tau+\bar\alpha\delta),\nonumber\\
\tau&=&c\cdot(\gamma\bar\beta+\alpha\bar\delta\tau)\nonumber.
\end{eqnarray}
It follows that $\gamma=\beta=0$, $\bar\alpha\delta$ is real, and
$c=(\bar\alpha\delta)^{-1}$.  So the action of (\ref{eq13}) is that
$P$ can be transformed to: $$\bar
cA^TPA=\frac1{\bar\alpha\delta}\left(\begin{array}{cc}\alpha&0\\0&\delta\end{array}\right)\left(\begin{array}{cc}a&b\\b&d\end{array}\right)\left(\begin{array}{cc}\alpha&0\\0&\delta\end{array}\right)=\left(\begin{array}{cc}a\alpha/\bar\delta&b\alpha/\bar\alpha\\b\alpha/\bar\alpha&d\delta/\bar\alpha\end{array}\right).$$
        When $b\ne0$, $\alpha$ can be chosen to rotate it onto the
        positive real axis.  Then, using real $\alpha$ and $\delta$,
        $a$ can be either scaled onto the unit circle or the origin;
        if $a=0$, then $d$ can be scaled onto the unit circle or the
        origin.  The resulting normal form is unique except that
        $(a,d)$ is equivalent to the pair $(-a,-d)$.

If $b=0$ and $a\ne0$, then $\alpha$ and $\delta$ can be chosen to
transform $a$ into $1$, leaving $d\in\co$ as an invariant.  If
$b=a=0$, $\alpha$ and $\delta$ can be chosen to transform $d$ into $1$
or $0$.

Case (2b).  For $N=\left(\begin{array}{cc}0&1\\0&0\end{array}\right)$,
  the $\tau=0$ case of Proposition \ref{thm4.3}, if $N=c\bar A^TNA$,
  then $c$ does not have to be real, but the calculation is similar to
  the previous (2a) case.  The $\tau=0$ analogue of (\ref{eq16})
  implies $\gamma=\beta=0$ and $c=(\bar\alpha\delta)^{-1}$.  So the
  action of (\ref{eq13}) is that $P$ can be transformed to: $$\bar
  cA^TPA=(\alpha\bar\delta)^{-1}\left(\begin{array}{cc}\alpha&0\\0&\delta\end{array}\right)\left(\begin{array}{cc}a&b\\b&d\end{array}\right)\left(\begin{array}{cc}\alpha&0\\0&\delta\end{array}\right)=\bar\delta^{-1}\left(\begin{array}{cc}a\alpha&b\delta\\b\delta&d\frac{\delta^2}{\alpha}\end{array}\right).$$
  When $b\ne0$, $\delta$ can be chosen to rotate it onto the positive
  real axis.  Then, using real $\delta$ and a complex number $\alpha$,
  $d$ can be transformed to $1$, leaving $a\in\co$ as an invariant, or
  to $0$ and then $a$ can be transformed to $1$ or $0$.

If $b=0$ and $d\ne0$, then $\alpha$ and $\delta$ can be chosen to
transform $d$ into $1$ and $a$ to a non-negative invariant.  If
$b=d=0$, $\alpha$ and $\delta$ can be chosen to transform $a$ into $1$
or $0$.

Case (3).  For $N=\left(\begin{array}{cc}0&1\\1&i\end{array}\right)$,
  the group of
  $\left(c,\left(\begin{array}{cc}\alpha&\beta\\\gamma&\delta\end{array}\right)\right)$
  such that $N=c\bar A^TNA$ is exactly the set with $\gamma=0$,
  $\alpha=\delta\ne0$, $\alpha\bar\beta+\bar\alpha\beta=0$, and
  $c|\alpha|^2=1$.  Then $$\bar
  cA^TPA=\frac1{|\alpha|^2}\left(\begin{array}{cc}a\alpha^2&a\alpha\beta+b\alpha^2\\a\alpha\beta+b\alpha^2&a\beta^2+2b\alpha\beta+d\alpha^2\end{array}\right).$$
  If $a\ne0$, then $\beta$ can eliminate $b$, $\alpha$ can rotate $a$
  to the positive real axis, and the complex number in the $d$
  position is an invariant.  If $a=0$ and $b\ne0$, then $\beta$ can
  eliminate $d$ and $\alpha$ can rotate $b$ to the positive real axis.
  If $a=b=0$, then $\alpha$ can rotate $d$ to the non-negative real
  axis.

Case (4).  For $N=\left(\begin{array}{cc}1&0\\0&0\end{array}\right)$,
  the group of
  $\left(c,\left(\begin{array}{cc}\alpha&\beta\\\gamma&\delta\end{array}\right)\right)$
  such that $N=c\bar A^TNA$ is exactly the set with $\beta=0$,
  $\delta\ne0$, and $c|\alpha|^2=1$.  The entry $d$ of $P$ can be
  normalized to $1$ or $0$.  In the $d=1$ case, $\gamma$ can eliminate
  any $b$ entry, and $\alpha$ can rotate $a$ onto the non-negative
  real axis.  In the $d=0$ case, if $b\ne0$, then $(a,b)$ can be
  normalized to $(0,1)$; if $b=0$, then $\alpha$ can rotate $a$ onto
  the non-negative real axis.

Case (5).  Under an arbitrary congruence transformation, the rank is
the only invariant of a complex symmetric matrix $P$ under a
congruence transformation $A^TPA$.  The three normal forms are:
  \begin{itemize} \item
  $P=\left(\begin{array}{cc}1&0\\0&1\end{array}\right)$; \item
      $P=\left(\begin{array}{cc}1&0\\0&0\end{array}\right)$; \item
        $P=\left(\begin{array}{cc}0&0\\0&0\end{array}\right)$.
  \end{itemize}

All these results on normal forms for the matrix pairs $(N,P)$ are
summarized by Table 1 in Section \ref{sum}.

\subsubsection{An alternative quadratic normal form}

Instead of choosing to normalize the $R$ matrix first in (\ref{eq6}),
we could have chosen to normalize the symmetric matrix $S$.  Then the
calculations start off in a simpler way, since first applying the congruence transformation to the complex symmetric matrix $S$, the three normal forms for $\bar A^TS\bar A$ are exactly as in the above Case (5).

In the first case, $S={\mathbbm 1}$, the normalization problem for
$(R,S)$ reduces to the problem of finding a normal form for $R$ under
the action $R\mapsto c\bar A^TRA$, where $(c,A)$ satisfies $c\bar
A^T\bar A={\mathbbm 1}$.  This is the generic form of $S$, so one
still expects five continuous real parameters in any collection of
representative matrix pairs, but we do not attempt to find such normal
forms.

In the second case, the pair $(c,A)$ stabilizing $S$ and acting on $R$
satisfies $$c\bar
A^T\left(\begin{array}{cc}1&0\\0&0\end{array}\right)\bar
A=\left(\begin{array}{cc}1&0\\0&0\end{array}\right),$$ leading to
another normal form problem for $R$ which we again do not pursue.

In the last case, where $S$ is the zero matrix, the $R$ matrix can
then be put into one of the normal forms $N$ from Proposition
\ref{thm4.3}.

For the special case where $R$ is Hermitian, \cite{iz} gives a list of
$2\times2$ normal forms for $(R,S)$, following this approach of
normalizing $S$ first.

\subsubsection{One example of a cubic normal form}\label{cnf}

With $M$ in standard position and the quadratic part of the defining
function in normal form (\ref{eq12}), we can consider its cubic terms.
In the
expansion $$z_3=(\z_1,\z_2)N\left(\begin{array}{c}z_1\\z_2\end{array}\right)+\rp\left((z_1,z_2)P\left(\begin{array}{c}z_1\\z_2\end{array}\right)\right)+e_3(z,\z)+e(z,\z),$$
    $e_3(z,\z)$ is the cubic part, and $e(z,\z)=O(4)$.  In particular,
\begin{eqnarray}
  e_3&=&e^{3000}z_1^3+e^{2100}z_1^2\z_1+e^{1200}z_1\z_1^2+e^{0300}\z_1^3\label{eq29}\\
  &&+e^{2010}z_1^2z_2+e^{1110}z_1\z_1z_2+e^{0210}\z_1^2z_2+e^{2001}z_1^2\z_2\nonumber\\
  &&+e^{1101}z_1\z_1\z_2+e^{0201}\z_1^2\z_2+e^{1020}z_1z_2^2+e^{0120}\z_1z_2^2\nonumber\\
  &&+e^{1011}z_1z_2\z_2+e^{0111}\z_1z_2\z_2+e^{1002}z_1\z_2^2+e^{0102}\z_1\z_2^2\nonumber\\
  &&+e^{0030}z_2^3+e^{0021}z_2^2\z_2+e^{0012}z_2\z_2^2+e^{0003}\z_2^3.\nonumber
\end{eqnarray}
The holomorphic coordinate changes that fix the origin and preserve
the standard position of $M$ are of the form
\begin{eqnarray}
  \tz_1&=&c_{11}z_1+c_{12}z_2+c_{13}z_3+p_1^{20}z_1^2+p_1^{11}z_1z_2+p_1^{02}z_2^2\label{eq27}\\
  \tz_2&=&c_{21}z_1+c_{22}z_2+c_{23}z_3+p_2^{20}z_1^2+p_2^{11}z_1z_2+p_2^{02}z_2^2\nonumber\\
  \tz_3&=&c_{33}z_3+p_3^{200}z_1^2+p_3^{110}z_1z_2+p_3^{020}z_2^2+p_3^{101}z_1z_3+p_3^{011}z_2z_3\nonumber\\
       &&+p_3^{300}z_1^3+p_3^{210}z_1^2z_2+p_3^{120}z_1z_2^2+p_3^{030}z_2^3.\nonumber
\end{eqnarray}
The linear coefficients are as in (\ref{eq9}); the $p$ coefficients
are from (\ref{eq1}).  Assigning to each monomial
$z_1^{\alpha_1}z_2^{\alpha_2}z_3^{\alpha_3}$ a ``weight''
$\alpha_1+\alpha_2+2\alpha_3$, including terms in (\ref{eq27}) of
higher weight, such as $z_3^2$, would not contribute any changes to
the quadratic or cubic coefficients of the defining function in the
$\tz$ coordinates.  The effect of a coordinate change (\ref{eq27}) on
the cubic part $e_3$ depends on the quadratic coefficients from $N$
and $P$.  If the quadratic part has already been put into a normal
form, then in attempting to find a normal form for the cubic terms,
one would want to use only holomorphic transformations that preserve
the quadratic normal form, and this subgroup of transformations also
depends on the coefficient matrices $N$ and $P$.  The comprehensive
problem of finding cubic normal forms for every equivalence class of
CR singularities seems to be difficult, so we consider just one
special but interesting case.

Let $M$ be given by this defining equation in standard position and
with quadratic normal form as in Case (1b) from Subsection \ref{qnf}:
\begin{equation}\label{eq26}
  z_3=z_1\z_1+\gamma_1\cdot(z_1^2+\z_1^2)+z_2\z_2+\gamma_2\cdot(z_2^2+\z_2^2)+e_3(z,\z)+e(z,\z).
\end{equation}
The above expression is in a quadratically flat normal form: the
quadratic part is real valued with $0\le\gamma_1\le\gamma_2$, and
$e_3$ is as above.  Now, assume further that $0<\gamma_1<\gamma_2$,
with neither $\gamma_1$ nor $\gamma_2$ equal to $\frac12$ or $1$.

It follows from the normal form result of \cite{mw} (as in Subsection
\ref{mn2}) that there is a transformation of the form
\begin{eqnarray}
  \tz_1&=&z_1+c_{13}z_3+p_1^{20}z_1^2\label{eq30}\\
  \tz_2&=&z_2\nonumber\\
  \tz_3&=&z_3+p_3^{101}z_1z_3+p_3^{300}z_1^3\nonumber
\end{eqnarray}
that eliminates the terms
$e^{3000}z_1^3+e^{2100}z_1^2\z_1+e^{1200}z_1\z_1^2+e^{0300}\z_1^3$
from $e_3$, without changing the quadratic part.  Similarly, a
transformation of $z_2$, $z_3$ can eliminate the cubic terms depending
only on $z_2$, $\z_2$, without changing the quadratic part or
re-introducing the cubic terms in $z_1$, $\z_1$.  This leaves twelve
monomials in (\ref{eq29}) with complex coefficients.

Among the transformations (\ref{eq27}), the subgroup preserving this
partial normal form is given by:
\begin{eqnarray}
  \tz_1&=&c_{11}z_1+p_1^{11}z_1z_2+p_1^{02}z_2^2\label{eq28}\\
  \tz_2&=&c_{22}z_2+p_2^{20}z_1^2+p_2^{11}z_1z_2\nonumber\\
  \tz_3&=&c_{33}z_3+p_3^{210}z_1^2z_2+p_3^{120}z_1z_2^2,\nonumber
\end{eqnarray}
again, omitting terms of higher weight.  The linear coefficients
$c_{11}$, $c_{22}$, $c_{33}$ must be real, and satisfy
$c_{33}=c_{11}^2=c_{22}^2$.  The six complex coefficients of $\vec p$
can be arbitrary, and then after a coordinate change of the form
(\ref{eq28}), the defining equation becomes:
\begin{eqnarray}
  \tz_3&=&\tz_1\btz_1+\gamma_1\cdot(\tz_1^2+\btz_1^2)+\tz_2\btz_2+\gamma_2\cdot(\tz_2^2+\btz_2^2)+\tilde
  e_3(\tz,\btz)+e(\tz,\btz),\nonumber\\ \tilde
  e_3&=&\left(\frac{e^{1110}}{c_{22}}-\frac{p_1^{11}}{c_{11}c_{22}}\right)\tz_1\btz_1\tz_2+\left(\frac{e^{1101}}{c_{22}}-\frac{\overline{p_1^{11}}}{c_{11}c_{22}}\right)\tz_1\btz_1\btz_2\nonumber\\
  &&+\left(\frac{e^{2001}}{c_{22}}-\frac{p_2^{20}}{c_{11}^2}\right)\tz_1^2\btz_2+\left(\frac{e^{0210}}{c_{22}}-\frac{\overline{p_2^{20}}}{c_{11}^2}\right)\btz_1^2\tz_2\label{eq32}\\
  &&+\left(\frac{e^{1011}}{c_{11}}-\frac{p_2^{11}}{c_{11}c_{22}}\right)\tz_1\tz_2\btz_2+\left(\frac{e^{0111}}{c_{11}}-\frac{\overline{p_2^{11}}}{c_{11}c_{22}}\right)\btz_1\tz_2\btz_2\nonumber\\
  &&+\left(\frac{e^{1002}}{c_{11}}-\frac{\overline{p_1^{02}}}{c_{22}^2}\right)\tz_1\btz_2^2+\left(\frac{e^{0120}}{c_{11}}-\frac{p_1^{02}}{c_{22}^2}\right)\btz_1\tz_2^2\nonumber\\
  &&+\left(\frac{e^{2010}}{c_{22}}-\frac{2\gamma_1p_1^{11}}{c_{11}c_{22}}-\frac{2\gamma_2p_2^{20}}{c_{11}^2}+\frac{p_3^{210}}{c_{11}^2c_{22}}\right)\tz_1^2\tz_2\nonumber\\
  &&+\left(\frac{e^{0201}}{c_{22}}-\frac{2\gamma_1\overline{p_1^{11}}}{c_{11}c_{22}}-\frac{2\gamma_2p_2^{20}}{c_{11}^2}\right)\btz_1^2\btz_2\nonumber\\
  &&+\left(\frac{e^{1020}}{c_{11}}-\frac{2\gamma_1p_1^{02}}{c_{22}^2}-\frac{2\gamma_2p_2^{11}}{c_{11}c_{22}}+\frac{p_3^{120}}{c_{11}c_{22}^2}\right)\tz_1\tz_2^2\nonumber\\
  &&+\left(\frac{e^{0102}}{c_{11}}-\frac{2\gamma_1\overline{p_1^{02}}}{c_{22}^2}-\frac{2\gamma_2\overline{p_2^{11}}}{c_{11}c_{22}}\right)\btz_1\btz_2^2.\nonumber
\end{eqnarray}
In general, there are not enough parameters in (\ref{eq28}) to put
these twelve coefficients into a sparse normal form.  We turn to yet a
further special case.

Suppose that after $M$ has been partially normalized as in
(\ref{eq30}), so that
$$e^{3000}=e^{2100}=e^{1200}=e^{0300}=e^{0030}=e^{0021}=e^{0012}=e^{0003}=0,$$
eight of the remaining twelve coefficients satisfy the following
conditions:
\begin{eqnarray}
  e^{1110}&=&\overline{e^{1101}}\label{eq31}\\
  e^{2001}&=&\overline{e^{0210}}\nonumber\\
  e^{1011}&=&\overline{e^{0111}}\nonumber\\
  e^{1002}&=&\overline{e^{0120}}.\nonumber
\end{eqnarray}
This condition holds if (but not only if) the partially normalized
$e_3$ is real valued, so the defining function of $M$ is in a $3$-flat
normal form.  Then, by inspection of (\ref{eq32}), there is a
transformation with complex coefficients $p_1^{02}$, $p_1^{11}$,
$p_2^{11}$, $p_2^{20}$ that can normalize the all the $e^{1110}$,
\ldots, coefficients in (\ref{eq31}) to $0$.  A transformation using
$p_3^{120}$ and $p_3^{210}$ can then change the $e^{2010}$ and
$e^{1020}$ coefficients to any value, in particular, to the complex
conjugates of $e^{0201}$ and $e^{0102}$, so that $e_3$ can be brought
to the following real valued normal form:
\begin{equation}\label{eq33}
  \tilde
  e_3=\overline{e^{0201}}\tz_1^2\tz_2+e^{0201}\btz_1^2\btz_2+\overline{e^{0102}}\tz_1\tz_2^2+e^{0102}\btz_1\btz_2^2.
\end{equation}
The coefficients $e^{0201}$, $e^{0102}$ are not invariants since a
real linear re-scaling by $c_{11}$ and $c_{22}=\pm c_{11}$ is still
possible.

There is a different, more useful, statement about the conditions
under which the above normal form can be achieved:
\begin{thm}\label{prop2.3}
  Suppose the real $4$-manifold $M$ in $\co^3$ is in standard position
  with defining equation in the quadratically flat normal form
  \mbox{\rm (\ref{eq26})} with $0<\gamma_1<\gamma_2$, and neither
  $\gamma_1$ nor $\gamma_2$ equal to $\frac12$, and with cubic terms
  $e_3$ as in \mbox{\rm (\ref{eq29})} that satisfy the conditions
\begin{eqnarray}
  e^{1200}&=&\overline{e^{2100}}\label{eq35}\\
  e^{0021}&=&\overline{e^{0012}}\label{eq36}\\
  e^{1110}&=&\overline{e^{1101}}\nonumber\\
  e^{2001}&=&\overline{e^{0210}}\nonumber\\
  e^{1011}&=&\overline{e^{0111}}\nonumber\\
  e^{1002}&=&\overline{e^{0120}}.\nonumber
\end{eqnarray}
  Then there is a holomorphic coordinate change \mbox{\rm
    (\ref{eq27})} that puts the cubic part into the $3$-flat normal
  form $\tilde e_3$ \mbox{\rm (\ref{eq33})}.
\end{thm}
\begin{proof}
  The last four out of the above six conditions are copied from
  (\ref{eq31}).  The proof of the Proposition is to proceed as above,
  first eliminating the cubic terms in $z_1,\z_1$ only.  The condition
  (\ref{eq35}) allows this to be done by a transformation of the form
  (\ref{eq30}), but with $p_3^{101}=0$.  (This is possible even in the
  $\gamma_1=1$ case, cf.\ Example 5.6 of \cite{mem}.)  Similarly, the
  cubic terms in $z_2,\z_2$ only can be eliminated by a transformation
  of the form
\begin{eqnarray}
  \tz_1&=&z_1\nonumber\\
  \tz_2&=&z_2+c_{23}z_3+p_2^{02}z_2^2\nonumber\\
  \tz_3&=&z_3+p_3^{030}z_2^3,\label{eq37}
\end{eqnarray}
  without re-introducing any of the previously eliminated terms.
  Again, the condition (\ref{eq36}) (and $\gamma_2\ne\frac12$) means
  that a term $p_3^{011}z_2z_3$ is not needed in (\ref{eq37}).  These
  first two steps may alter the numerical values of the remaining
  coefficients, but the claim is that if the cubic coefficients
  $e^{1110}$, \ldots, $e^{0120}$ satisfy the reality conditions
  (\ref{eq31}) at the start of the process, then the new corresponding
  coefficients will continue to satisfy those conditions.  The
  calculation to verify this is straightforward but omitted; however,
  the assumption $p_3^{101}=p_3^{011}=0$ is crucial: if either were
  nonzero, cubic terms not satisfying (\ref{eq31}) could appear in
  the new defining equation.  The rest of the normalization proceeds
  exactly as above.
\end{proof}
The conclusion of the above Theorem holds in particular when $M$ has
the specified quadratic normal form and is also $3$-flat.

\subsection{Flatness in higher dimensions}\label{hf}

\

  Proposition \ref{thm4.3} showed that for the $m=4$, $n=3$ case
  considered in Subsection \ref{mn43}, $R$ cannot, in general, be put
  into a Hermitian normal form.  Similarly for higher dimensions
  $n>2$, $m=2n-2$, quadratic flatness is a non-generic property for
  codimension 2 manifolds with CR singularities.  Higher degree
  flatness is even more non-generic.

  Consider the quadratically flat case, where $M$ is in standard
  position in $\co^n$ and the coefficient matrix $R$ happens to be
  Hermitian symmetric.  This property of $R$ is preserved by the
  action of $(c_{n,n},A)$ from (\ref{eq9}) if and only if $c_{n,n}$ is
  real --- unless $R$ is the zero matrix, in which case there is no
  such condition on $c_{n,n}$.

  For a nonzero Hermitian matrix $R$, the transformations (\ref{eq3})
  (or (\ref{eq13})) of the defining function that preserve the
  property of being in a quadratically flat normal form have the
  following action: $R\mapsto c_{n,n}\bar A^TRA$, a congruence
  transformation of $R$, followed by real scalar multiplication.  The
  Hermitian property of $R$ is preserved, along with its rank.
  Congruence transformations also preserve the signature $(p,q)$: the
  number of positive, and negative, eigenvalues; $R$ is ``definite''
  if $p=n$ or $q=n$.  Multiplying by a negative scalar $c_{n,n}$
  interchanges the $p$ and $q$ quantities.  The rank $\rho(R)$, the
  number $\sigma(R)=|p-q|$ (from which $p$ and $q$ can be recovered,
  modulo switching), and the property of definiteness (or
  indefiniteness) are invariants of $R$ under the action of this
  transformation group.

  If $R$ is the zero matrix, then $c_{n,n}R$ is still Hermitian even
  if $c_{n,n}$ is not a real number.  $M$ is quadratically flat: the
  quadratic part of $h(z,\z)$ can be made real valued by a transformation of
  the form (\ref{eq4}).

  Including both cases $R\ne0$, $R=0$, we can conclude that if $M$ has
  a defining function $h(z,\z)$ in a quadratically flat normal form, then
  any holomorphic coordinate change that preserves the property that
  $h(z,\z)$ is in a quadratically flat normal form must leave invariant the
  rank $\rho(R)$ and the number $\sigma(R)=|p-q|$.  In particular,
  these quantities are also invariants of a defining function in a
  ${\bf d}$-flat normal form, or which is real valued (the
  holomorphically flat case), under holomorphic coordinate changes
  that preserve the flatness property of the defining function.  The
  determinant of $R$ is real and transforms as: $\det(c_{n,n}\bar
  A^TRA)=c_{n,n}^{n-1}|\det(A)|^2\det(R)$, so the sign of $\det(R)$ is
  also an invariant if $n$ is odd.

  For arbitrary dimensions $m=2n-2$, $n\ge2$, if the
  $(n-1)\times(n-1)$ coefficient matrix $R$ in (\ref{eq6}) is
  Hermitian and definite (so that $R$ transforms by congruence to the
  identity matrix $\mathbbm1$), the result of Takagi is that then
  there is a unitary coordinate change diagonalizing $S$, with real
  entries $0\le\gamma_1\le\ldots\le\gamma_{n-1}$ on the diagonal.
  These numbers are called ``generalized Bishop invariants'' by
  \cite{hy}, since the $0\le\gamma_1<\infty$ normal form (\ref{eq34})
  from Subsection \ref{mn2} is the $n=2$ special case.  The definite
  normal form from Case (1b) of Subsection \ref{qnf} is the $n=3$
  case.  The $R={\mathbbm1}$, $S=0_{(n-1)\times(n-1)}$ case, where
  $z_n=z_1\z_1+\ldots+z_{n-1}\z_{n-1}+O(3)$, is considered by
  \cite{hy}.

  For $P$ complex symmetric and $R$ Hermitian symmetric, but not
  necessarily definite, a description of canonical representatives for
  the equivalence relation $(R,P)\sim(\bar A^TRA,A^TPA)$ is given by
  \cite{e}.  Allowing scalar multiplication, as in (\ref{eq13}), with
  real $c_{n,n}$ as above, is different only in that $(R,P)$ is
  equivalent to $(-R,-P)$ under the action of (\ref{eq13}), while they
  may be inequivalent according to \cite{e}.  In the $2\times2$ case,
  the choices made by \cite{iz} and \cite{e} for canonical forms are
  different from those in the calculations of Subsection \ref{qnf},
  Cases (1b), (1c), (4), (5).  However, it is straightforward to check
  how our list of normal forms, as summarized in Examples \ref{ex7.3}
  -- \ref{ex7.7}, corresponds (modulo scalars) to the systems of
  canonical forms in \cite{iz}, \cite{e}.

\section{Topological considerations}\label{tc}

We recall some well-known general facts on Grassmannian manifolds, and
also recall from \cite{lai} some topological properties of real
$m$-submanifolds of complex $n$-manifolds: the local property of
``general position'' for real submanifolds, and also global properties
measured by characteristic classes.  The Grassmannian constructions
and formulas of \cite{lai}, in the special case of the topology of
real surfaces immersed in complex surfaces ($m=n=2$), are also
reviewed by \cite{BF}, \cite{bohr}, \cite{for}, \cite{abklr} \S 8.5.3.
In this Section, we work with connected, $m$-dimensional, smooth
manifolds $M$, not necessarily real analytic.

\subsection{The Grassmannian construction}\label{tgc}

\ 

For $0<m<2n$, let $G(m,\re^{2n})$ denote the Grassmannian manifold
of $m$-dimensional real linear subspaces of $\re^{2n}$.  The real
dimension of $G(m,\re^{2n})$ is $m\cdot(2n-m)$.

Similarly, let $SG(m,R^{2n})$ denote the manifold of oriented
$m$-dimensional subspaces of $\re^{2n}$.  Its real dimension is also
$m\cdot(2n-m)$. There is a two-to-one covering map $\ff:SG(m,\re^{2n})\to
G(m,\re^{2n})$ given by forgetting the orientation, and an involution
$\rr:SG(m,\re^{2n})\to SG(m,\re^{2n})$ given by reversing the
orientation.

For any immersion $\iota$ of a real $m$-manifold $M$ into $\re^{2n}$,
we can define the Gauss map $\bg:M\to G(m,\re^{2n}):x\mapsto
T_{\iota(x)}M$.  If $M$ has an orientation, then the corresponding map
is denoted $\bg_s:M\to SG(m,\re^{2n})$.

  When $m$ is even and there is a ``complex structure operator,'' $J$,
  a real linear map $\re^{2n}\to\re^{2n}$ such that $J\circ
  J=-{\mathbbm 1}$, then $G(m,\re^{2n})$ contains $\cc$, the set of
  $J$-invariant subspaces.  $\cc$ is a submanifold: it is the image in
  $G(m,\re^{2n})$ of the inclusion embedding of the complex
  Grassmannian manifold $\co G(m/2,\co^n)$ of complex subspaces in
  $\co^n$.  

  In the case $m=2n-2$, the manifold $\co G(n-1,\co^n)$ is
  homeomorphic to the complex projective space $\co P^{n-1}$, and the
  real dimension of $\cc$ is $2(n-1)$, half the dimension $4(n-1)$ of
  $G(2n-2,\re^{2n})$.  The inverse image $\ff^{-1}(\cc)\subseteq
  SG(2n-2,\re^{2n})$ is a disjoint union $\cc^+\cup\cc^-$, where
  $\cc^+$ is the set of oriented subspaces whose orientation agrees
  with that given by the complex structure, and $\cc^-$ is the set of
  subspaces where these orientations are opposite (so they could be
  called ``anticomplex'' subspaces).  Each component $\cc^+$ and
  $\cc^-$ is the image of an embedding of $\co G(n-1,\co^n)$, with
  real dimension, and codimension, equal to $2(n-1)$.  The $\cc^+$
  submanifold has a natural orientation.

For $0<m<2n$ and $0\le\jl\le n$, given an immersion of an
$m$-submanifold $\iota:M\to\re^{2n}=\co^n$, let $N_\jl$ denote the
subset $$N_\jl=\{x\in M:\dim_\re(T_xM\cap J\cdot T_xM)\ge2\jl\},$$ so
at a point $x\in N_\jl$, the tangent space $T_xM$ contains a
$J$-invariant subspace of real dimension $2\jl$, which is a complex
subspace of complex dimension $\jl$ in $\co^n$.

For $m=2n-2$, $N_{n-2}=M$ and $N_{n-1}=\bg^{-1}(\cc)$ is the CR
singular locus.  The immersion is in ``general position'' at $x\in M$
if the gauss map $\bg:M\to G(2n-2,\re^{2n})$ meets $\cc$ transversely
at $\bg(x)$.  $M$ is trivially in general position at all its CR
generic points, and the immersion is said to be in general position if
it is in general position at every point.  Then, by counting
dimensions, an immersion in general position has only isolated CR
singular points, and if $M$ is compact, then the CR singular locus
$N_{n-1}$ will be a finite set.

If $M$ has an orientation, then the CR singular locus is the same set
as $(\ff\circ
\bg_s)^{-1}(\cc)=\bg_s^{-1}(\ff^{-1}(\cc))=\bg_s^{-1}(\cc^+\cup\cc^-)$,
which is a disjoint union $\bg_s^{-1}(\cc^+)\cup \bg_s^{-1}(\cc^-)$.
So, $N_{n-1}=N_{n-1}^+\cup N_{n-1}^-$, where
$N_{n-1}^+=\bg_s^{-1}(\cc^+)$ and $N_{n-1}^-=\bg_s^{-1}(\cc^-)$.  The
local and global notions of ``general position'' as defined previously
are equivalent to the analogous transverse meeting of $\bg_s$ with
$\cc^+\cup\cc^-$ in $SG(2n-2,\re^{2n})$.  At each point of $N_{n-1}^+$
in general position, there is an oriented intersection number, $\pm1$,
of $\bg_s$ with the oriented submanifold $\cc^+$: the intersection
number at $\bg_s(x)$ will be denoted the index, $ind(x)$.

To define the intersection index at a point $x$ of $N_{n-1}^-$, \cite{lai}
makes a choice of orientation for $\cc^-$ which is opposite to that
induced by $\rr:\cc^+\to\cc^-$.  Then $ind(x)$, the intersection number
of $\bg_s$ with $\cc^-$ at $\bg_s(x)$, is equal to the
intersection number of $\rr\circ \bg_s$ with $\cc^+$ at
$\rr(\bg_s(x))$.  Equivalently, if $M^-$ denotes the manifold $M$
with its orientation reversed, with gauss map $\bg_s^\prime:M^-\to
SG(2n-2,\re^{2n})$, then $ind(x)$ is equal to the intersection number
of $\bg_s^\prime$ with $\cc^+$ at $\bg_s^\prime(x)$.

\subsection{Bundle maps}

\ 

More generally, let $M$ be a smooth, oriented manifold with real
dimension $2n-2$, and let $F$ be a smooth, oriented real vector bundle
over $M$ with $2n$-dimensional fibers.  Then the space of oriented
real $(2n-2)$-subspaces of fibers of $F$ forms a ``grassmann bundle''
$SG(2n-2,F)$ over $M$ --- on a local coordinate patch $U$ of $M$ where
$F$ can be trivialized as $U\times\re^{2n}$, the grassmann bundle is
of the form $U\times SG(2n-2,\re^{2n})$.  Suppose $F$ admits a smooth
complex structure operator $J$.  Then for each point $x\in M$, the
$J_x$-invariant subspaces form sets $\cc^+_x$ and $\cc^-_x$ in the
fiber over $x$, giving a pair of smooth bundles $\cc^+$, $\cc^-$ of
complex grassmannians over $M$; each total space has codimension
$2n-2$ in the total space $SG(2n-2,F)$.  If $T$ is another smooth,
oriented real vector bundle over $M$, with $(2n-2)$-dimensional
fibers, then a nonsingular bundle map $\mu:T\to F$ induces a section
$\bg_\mu:M\to SG(2n-2,F):x\mapsto\mu(T_x)$, generalizing the gauss
map.  The transverse intersection of $\bg_\mu$ with $\cc^+$ and
$\cc^-$ defines a notion of ``general position'' for $T$, and the
intersection numbers define the index of generally isolated points $x$
where $\bg_\mu(x)=\mu(T_x)$ is a $J_x$-invariant subspace of $F_x$.

A special case of this bundle construction is where $T$ is the tangent
bundle $TM$ of $M$, and $\iota$ is an immersion of $M$ in an almost
complex manifold $\aaa$, with real dimension $2n$ and a smooth complex
structure operator $J$ on $T\aaa$.  Then the differential of $\iota$
defines a smooth bundle map $\mu:TM\to F=\iota^*T\aaa$, inducing
$\bg_\mu:M\to SG(2n-2,F)$.  A point $x\in M$ is a CR singular point of
$\iota$ if $\bg_\mu(x)=\mu(T_xM)$ is a $J$-invariant subspace of
$T_{\iota(x)}\aaa$.

If, additionally, $T\aaa$ admits a positive definite Riemannian metric
$g$, then there is an oriented real $2$-plane normal bundle $\nu M$
orthogonal to $TM$ in $\iota^*T\aaa$, and further, if $g$ has the
property that $J$ is an isometry with respect to $g$, then $T_xM$ is a
complex hyperplane in $\iota^*T_x\aaa$ if and only if $\nu_xM$ is a
complex line.  Such a metric $g$ can be chosen for any $(T\aaa,J)$,
although we will not be using it except to define the normal bundle.

\subsection{Characteristic class formulas}\label{ccf}

\ 

Following the notation of \cite{for}, for an immersion $\iota:M\to
\aaa$ (or bundle map $\mu:T\to F$) as in the previous Subsection,
we define index sums:
\begin{eqnarray*}
  I_+&=&\sum_{x\in
N_{n-1}^+}ind(x)\\
  I_-&=&\sum_{x\in N_{n-1}^-}ind(x).
\end{eqnarray*}
When $M$ is compact and the immersion (or bundle map) is in general
position, $I_+$ and $I_-$ are finite sums of $\pm1$ terms.  Then
$I_+$, $I_-$, $I_++I_-$, and $I_+-I_-$ are all invariants of the
homotopy class of $\iota$ (or $\mu$).  Reversing the orientation of
$M$ interchanges the values of $I_+$ and $I_-$, so $I_++I_-$ is the
same and $I_+-I_-$ has the opposite sign.

If, instead of following \cite{lai} as in Subsection \ref{tgc}, we
make the other choice of orientation for $\cc^-$, then $I_-$ has the
opposite sign and the quantities $I_++I_-$ and $I_+-I_-$ are switched.
It will be seen in some examples that this is the choice of
orientation that corresponds to index sums appearing in enumerative
formulas of \cite{Webster}, \cite{H-L1}, \cite{H-L2}, \cite{domrin2},
and \cite{me}.

Our notation for characteristic classes in cohomology is copied from
\cite{lai}: denote the euler class of the tangent bundle $TM$ by
$\Omega$, denote the euler class of the normal bundle $\nu M\to M$ of
the immersion by $\widetilde\Omega$, and denote the total chern class
of the pullback bundle $\iota^*(T\aaa,J)\to M$ by
$1+c_1+c_2+\ldots+c_n$.

\begin{prop}[\cite{lai}]
  For an immersion $\iota$ of a compact, oriented $(2n-2)$-manifold
  $M$ in general position in an almost complex $2n$-manifold $\aaa$,
  $$I_+=\int_M\frac12\left(\Omega+\sum_{r=0}^{n-1}\widetilde\Omega^{r}c_{n-1-r}\right).$$
  \boxx
\end{prop}
\begin{cor}
  With the orientation convention for $\cc^-$ as in \cite{lai},
  \begin{eqnarray*}
    I_-&=&\int_M\frac12\left(\Omega+\sum_{r=0}^{n-1}(-1)^{r+1}\widetilde\Omega^{r}c_{n-1-r}\right),\\
    I_++I_-&=&\int_M\left(\Omega+\sum_{k=0}^{\lfloor n/2\rfloor-1}\widetilde\Omega^{2k+1}c_{n-2k-2}\right),\\
    I_+-I_-&=&\int_M\sum_{k=0}^{\lfloor(n-1)/2\rfloor}\widetilde\Omega^{2k}c_{n-2k-1}.
  \end{eqnarray*}
\end{cor}
\begin{proof}
  $I_-$ is calculated by reversing the orientation of the tangent and
  normal bundles $TM$ and $\nu M$, which switches the sign of $\Omega$
  and $\widetilde\Omega$.  Then, applying the formula from the
  Proposition, integrating over $M^-$ gives the opposite of the
  integral over $M$.
\end{proof}

\begin{example}\label{ex3.3}
  When $\aaa=\co^n$, the chern classes are trivial, so for $M$
  immersed in general position,
  \begin{eqnarray*}
    I_+&=&\int_M\frac12\left(\Omega+\widetilde\Omega^{n-1}\right),\\
    I_-&=&\int_M\frac12\left(\Omega+(-1)^{n}\widetilde\Omega^{n-1}\right).
  \end{eqnarray*}
  If, in addition, $\iota$ is an embedding, then $\widetilde\Omega$ is
  the zero class, so $I_+=I_-=\frac12\upchi(M)$, where $\upchi(M)$ is the
  euler characteristic of $M$ (\cite{lai} Theorem 4.10).
\end{example}

\begin{example}
  For $n=2$, $\iota$ is an immersion of a compact, oriented real
  surface $M$ in a $4$-manifold $\aaa$ with an almost complex
  structure $J$.  Complex or anticomplex points with Lai's index $+1$
  are ``elliptic,'' or ``hyperbolic'' with index $-1$.  Again
  following the notation of \cite{for} and \cite{abklr}, if $e_+$
  (respectively, $e_-$) is the number of elliptic points with
  positively (negatively) oriented complex tangent spaces and $h_+$
  ($h_-$) is the number of positive (negative) hyperbolic points, and
  $e=e_++e_-$, $h=h_++h_-$, then for $M$ in general position,
  \begin{eqnarray*}
  I_+&=&e_+-h_+=\int_M\frac12\left(\Omega+c_1+\widetilde\Omega\right),\\
  I_-&=&e_--h_-=\int_M\frac12\left(\Omega-c_1+\widetilde\Omega\right),\\
  I_++I_-&=&e-h=\int_M\left(\Omega+\widetilde\Omega\right),\\
  I_+-I_-&=&(e_+-e_-)-(h_+-h_-)=\int_Mc_1.
  \end{eqnarray*}

  The last formula was also proved by \cite{Webster} and is a special
  case of a degeneracy locus formula of \cite{H-L2} and
  \cite{domrin2}.  When $\iota$ is an embedding into $\aaa=\co^2$,
  $\widetilde\Omega=0$, so the formulas are
  $e_+-h_+=e_--h_-=\frac12\upchi(M)$, which were known to \cite{bishop}
  and \cite{Wells} in the special cases where $M$ is an embedded
  sphere or torus.  Immersions with double points are considered by
  \cite{BF} and \cite{bohr}.
\end{example}

\begin{example}\label{ex3.5}
  For $n=3$, $\iota$ is an immersion of a compact, oriented real
  $4$-manifold $M$ in a $6$-manifold $\aaa$ with an almost complex
  structure $J$.  For $M$ in general position, we adapt from the
  previous Example the notation $e_+$, $h_+$ for counting elements of
  $N_2^+$, and $e_-$, $h_-$ for elements of $N_2^-$.
  \begin{eqnarray*}
  I_+&=&e_+-h_+=\int_M\frac12\left(\Omega+c_2+\widetilde\Omega
  c_1+\widetilde\Omega^2\right),\\
  I_-&=&e_--h_-=\int_M\frac12\left(\Omega-c_2+\widetilde\Omega
  c_1-\widetilde\Omega^2\right),\\
  I_++I_-&=&e-h=\int_M\left(\Omega+\widetilde\Omega c_1\right),\\
  I_+-I_-&=&(e_+-e_-)-(h_+-h_-)=\int_M\left(c_2+\widetilde\Omega^2\right).
  \end{eqnarray*}

  There are also formulas
  \begin{eqnarray}
    I_+-I_-&=&\int_M\left(c_2+p_1\nu M\right)\label{eq10}\\
    &=&\int_M\left(c_1^2-c_2-p_1TM\right),\label{eq11}
  \end{eqnarray}
  where $p_1$ is the first pontrjagin class of the normal bundle $\nu
  M$ or tangent bundle $TM$.  The equality (\ref{eq10}) follows from
  the well-known identity of characteristic classes
  $\widetilde\Omega^2=p_1\nu M$.  Formula (\ref{eq11}) is a special
  case of the degeneracy locus formulas of \cite{H-L2} and
  \cite{domrin2} --- the calculation establishing the equivalence of
  (\ref{eq10}) and (\ref{eq11}) appears in \cite{me}.
\end{example}

\begin{example}\label{ex3.6}
  Not every compact, oriented $4$-manifold can be immersed in
  $\re^6\cong\co^3$ --- for example, $\co P^2$ cannot (\cite{hirsch}).
  If such a $4$-manifold $M$ is immersed in general position in
  $\co^3$, then $I_++I_-=\upchi(M)$ and
  $$I_+-I_-=\int_Mp_1\nu M=\int_M-p_1TM=-p_1M,$$ the opposite of the
  first pontrjagin number of $M$ (\cite{H-L1}).  The number
  $\int_Mp_1\nu M$ is also three times the algebraic number of triple
  points of an immersion in general position (\cite{herbert}).  A
  compact, oriented $4$-manifold admitting a CR generic immersion in
  $\co^3$ must have $\upchi(M)=p_1M=0$; the converse problem, finding a
  sufficient condition for the existence of a CR generic immersion, is
  considered by \cite{jl}.
\end{example}

\begin{example}\label{ex3.7}
  Consider the complex projective space $\aaa=\co P^3$, and a
  nonsingular, degree ${\boldsymbol d}$ complex hypersurface $Y\subseteq\co P^3$.
  In terms of the hyperplane class $H$ in the cohomology ring of $\co
  P^3$, there are well-known formulas (\cite{gh} \S 1.1, 3.4, 4.6) for
  chern classes: $c_2(T\co P^3|_Y)=6H^2$, $c_1(T\co P^3|_Y)=4H$, and
  $c_1\nu Y={\boldsymbol d}H$.  Since $\int_YH^2={\boldsymbol d}$,
  \begin{eqnarray*}
    \int_Y\left(c_2(T\co P^3|_Y)+(c_1\nu Y)^2\right)&=&6{\boldsymbol d}+{\boldsymbol{d}}^3,\\
    \int_Y\left(c_2TY+c_1\nu Y\cdot c_1(T\co
    P^3|_Y)\right)&=&\upchi(Y)+4{\boldsymbol d}^2.
  \end{eqnarray*}
  Let $M$ be a smooth real submanifold of $\co P^3$ isotopic to $Y$,
  that is, homotopic through a family of smooth embeddings.  Then the
  chern classes of $Y$ pull back to classes on $M$, and if $M$ is in
  general position, then
  \begin{eqnarray*}
    \int_Y\left(c_2(T\co P^3|_Y)+(c_1\nu
    Y)^2\right)&=&\int_M\left(c_2(T\co
    P^3|_M)+\widetilde\Omega^2\right)\\ =6{\boldsymbol d}+{\boldsymbol d}^3&=&I_+-I_-,\\ \int_Y\left(c_2TY+c_1\nu
    Y\cdot c_1(T\co
    P^3|_Y)\right)&=&\int_M\left(\Omega+\widetilde\Omega\cdot c_1(T\co
    P^3|_M)\right)\\ =\upchi(M)+4{\boldsymbol d}^2&=&I_++I_-.
  \end{eqnarray*}
  by the formulas from Example \ref{ex3.5}.  These numbers are always
  positive; there is no CR generic submanifold of $\co P^3$ isotopic
  to a smooth complex hypersurface.  One also expects that
  sufficiently nearby perturbations of $Y$ would have only positively
  oriented complex tangents ($N_2^-=\mbox{\O}$), so $I_-=0$ and we
  recover the formula $\upchi(Y)=\upchi(M)={\boldsymbol d}^3-4{\boldsymbol d}^2+6{\boldsymbol d}$ (\cite{gh}
  p.\ 601).

  In particular, consider the holomorphic embedding given in
  homogeneous coordinates by ${\mathbf c}:\co P^2\to\co
  P^3:[z_0:z_1:z_2]\to[z_0:z_1:z_2:0]$, so the image $Y$ is a complex
  hyperplane with degree ${\boldsymbol d}=1$.  $\co P^2$ considered only
  as a smooth, oriented $4$-manifold (forgetting its complex structure)
  has $\upchi(\co P^2)=3$.  An embedding of $M=\co P^2$ as an oriented
  real submanifold in general position and isotopic to $\mathbf c$ has
  seven complex tangents, counted as an index sum with multiplicity
  according to either the $I_++I_-$ or $I_+-I_-$ sign convention.  In
  Section \ref{vge} we give a concrete example of such an isotopy from
  $\mathbf c$ to a smooth embedding $\co P^2\to\co P^3$ in general
  position with exactly seven CR singular points in $N_2^+$, each with
  index $=+1$: $I_+=e_+=7$, $h_+=I_-=e_-=h_-=0$.
\end{example}

\begin{rem}\label{rem3.8}
  If $\iota:\co P^2\to\co P^3$ is any immersion (or any continuous map
  homotopic to an immersion), then $\iota$ is homotopic to either the
  embedding $\mathbf c$ from the above Example, or the composite
  ${\mathbf c}\circ\kappa$, where $\kappa$ is the orientation-preserving
  involution $[z_0:z_1:z_2]\mapsto[\z_0:\z_1:\z_2]$ (\cite{thomas},
  \cite{lp}).  Equivalently, either $\iota$ or $\iota\circ\kappa$ is
  homotopic to $\mathbf c$.
\end{rem}

\subsection{Local coordinates for the Grassmannian}\label{lcg}

\ 

For $0<m<2n$, each element $v\in G(m,\re^{2n})$ is the image of some
linear map $\re^m\to\re^{2n}$ with standard matrix representation
$X_{2n\times m}$ of rank $m$, and any two linear maps with the same
image are right-equivalent ($X\sim XY$ for invertible $Y_{m\times
  m}$).  If $v^0$ is the $m$-plane $\{(v_1,\ldots,v_m,0,\ldots,0)^T\}$
in $\re^{2n}$, then it is the image of
$X^0=\left(\begin{array}{c}{\mathbbm1}_{m\times
    m}\\0\end{array}\right)_{2n\times m}$, and any elements
  sufficiently near $v^0$ are the image of some linear map whose
  matrix representation can be column-reduced to the form
  \begin{equation}\label{eq20}
  \left(\begin{array}{c}{\mathbbm1}_{m\times m}\\V_{(2n-m)\times
      m}\end{array}\right)_{2n\times m}.
  \end{equation}  The matrices $V$ in a
  neighborhood of the $(2n-m)\times m$ zero matrix form a local
  coordinate chart around $v^0$ in $G(m,\re^{2n})$.  The inverse image
  of a sufficiently small chart under $\ff:SG(m,\re^{2n})\to
  G(m,\re^{2n})$ gives a pair of charts in $SG(m,\re^{2n})$, one
  around each element of $\ff^{-1}(v^0)$, the $m$-plane with its two
  possible orientations.

In the case where $m$ is even and $\re^{2n}$ has coordinates
$(x_1,y_1,\ldots,x_n,y_n)^T$ and a complex structure operator
$$J_{2n\times
  2n}=\left(\begin{array}{ccccc}0&-1&&&\\1&0&&&\\&&\ddots&&\\&&&0&-1\\&&&1&0\end{array}\right),$$
  consider the $m$-plane
$$v^0=\{(x_1,y_1,\ldots,x_{m/2},y_{m/2},0,\ldots,0)^T\}.$$ Since $v^0$
  is $J$-invariant, $v^0\in\cc$.  Any element $v$ of the intersection
  of $\cc$ with the local coordinate chart around $v^0$ would have
  matrix representation
  $X=\left(\begin{array}{c}{\mathbbm1}\\V\end{array}\right)_{2n\times
      m}$, such that $J\cdot X\sim X$.  The equivalence
  \begin{eqnarray*}
    J\cdot\left(\begin{array}{c}{\mathbbm1}\\V\end{array}\right)&=&\left(\begin{array}{c}J_{m\times
          m}\\J_{(2n-m)\times(2n-m)}\cdot
        V\end{array}\right)\\
    \sim\left(\begin{array}{c}J\\J\cdot
          V\end{array}\right)\cdot(-J)_{m\times
            m}&=&\left(\begin{array}{c}{\mathbbm1}\\-J\cdot V\cdot
            J\end{array}\right)
  \end{eqnarray*}
   shows $v\in\cc\iff V=-J\cdot V\cdot J\iff V\cdot J=J\cdot V$, that
   is, $V$ is complex linear with respect to $J_{m\times m}$ and
   $J_{(2n-m)\times(2n-m)}$.

For example, in the $m=4$, $n=3$ case, the coordinates near $v^0$ in the
$8$-dimensional space $G(4,\re^6)$ are of the form
$V=\left(\begin{array}{cccc}a_1&a_2&a_3&a_4\\b_1&b_2&b_3&b_4\end{array}\right)$,
  and the elements of $\cc$ have coordinates $V$
  satisfying 
  \begin{eqnarray*}
    &&\left(\begin{array}{cc}0&-1\\1&0\end{array}\right)\cdot\left(\begin{array}{cccc}a_1&a_2&a_3&a_4\\b_1&b_2&b_3&b_4\end{array}\right)\\
    &=&\left(\begin{array}{cccc}a_1&a_2&a_3&a_4\\b_1&b_2&b_3&b_4\end{array}\right)\cdot\left(\begin{array}{cccc}0&-1&&\\1&0&&\\&&0&-1\\&&1&0\end{array}\right).
  \end{eqnarray*}
  This linear condition on $a_1,\ldots,b_4$ has a $4$-dimensional
  solution subspace $$\{b_1=-a_2,b_2=a_1,b_3=-a_4,b_4=a_3\}.$$

Now consider $v^0$ with its orientation as a complex subspace of
$(\re^6,J)$, so $v^0\in\cc^+\subseteq SG(4,\re^6)$.  Also, instead of
$2\times4$ matrices, we put the eight coordinate functions for
$SG(4,\re^6)$ near $v^0$ in a column vector format,
$(a_1,a_2,a_3,a_4,b_1,b_2,b_3,b_4)^T$.  Then, in this chart around
$v^0$, $\cc^+$ is the image of the linear map $\re^4\to\re^8$ with
matrix representation
$$\left(\begin{array}{cccc}1&0&0&0\\0&1&0&0\\0&0&1&0\\0&0&0&1\\0&-1&0&0\\1&0&0&0\\0&0&0&-1\\0&0&1&0\end{array}\right)=\left(\begin{array}{c}{\mathbbm1}_{4\times4}\\J_{4\times4}\end{array}\right).$$
Returning to the general situation of codimension $2$ in $\co^n$, the
above pattern still holds, so that $\cc^+$ is the column space of
$\left(\begin{array}{c}{\mathbbm1}\\J\end{array}\right)_{2(2n-2)\times(2n-2)}$
in the coordinate chart around $v^0$ in $SG(2n-2,\re^{2n})$.

\section{A reformulation of Garrity's transversality criterion}\label{atc}

Given a real $(2n-2)$-submanifold of $\co^n$ with a CR singular point,
we now consider the problem of determining from the local defining
equation whether $M$ is in general position, as defined in Section
\ref{tc}, and if so, how the local defining equation determines the
intersection index.  The transversality problem was also considered by
\cite{g} for real $(2n-2)$-submanifolds of $\co^n$, using different
methods but arriving at an equivalent result.  Our result (Theorem
\ref{thm5.1}) relates transversality to an expression in terms of the
coefficient matrix notation from Section \ref{sec1}.

\subsection{A determinantal formula}\label{adf}

\ 

We begin by assuming $M$ is in standard position, given an orientation
agreeing with the orientation of its complex tangent space at the
origin.  $M$ can be described by a complex implicit equation
(\ref{eq0}) in a neighborhood $\Delta$ of $\vec0$:
\begin{eqnarray}
  0&=&z_n-h(z,\z)=z_n-\left(z^TQz+\z^TRz+\z^TS\z+e(z,\z)\right).\label{eq21}
\end{eqnarray}
If we now consider two real functions,
$f^1(x_1,y_1,\ldots,x_{n-1},y_{n-1})=\rp(h(z,\z))$ and
$f^2(x_1,\ldots,y_{n-1})=\ip(h(z,\z))$, then the real $(2n-2)$-manifold $M$
has a local parametrization $\pi$ with domain
$\ddd\subseteq\re^{2n-2}$ and embedding target $\re^{2n}$:
\begin{equation}\label{eq24}
  \pi:(x_1,\ldots,y_{n-1})^T\mapsto(x_1,\ldots,y_{n-1},f^1,f^2)^T.
\end{equation}
The differential of this map assigns to each point $z\in\ddd$ the
linear map from $\re^{2n-2}$ to $T_{\pi(z)}M\subseteq\re^{2n}$; this
linear map has matrix representation:
$$\left(\begin{array}{ccc}&&\\&{\mathbbm1}_{(2n-2)\times(2n-2)}&\\&&\\\frac{df^1}{dx_1}&\cdots&\frac{df^1}{dy_{n-1}}\\\frac{df^2}{dx_1}&\cdots&\frac{df^2}{dy_{n-1}}\end{array}\right)_{(2n)\times(2n-2)}.$$
This matrix is already in the form (\ref{eq20}), so in
the local coordinate systems $\pi$ for $M$ and $V$ for
$SG(2n-2,\re^{2n})$, the oriented gauss map has the form
$$\bg_s:(x_1,y_1,\ldots,x_{n-1},y_{n-1})^T\mapsto\left(\begin{array}{ccc}\frac{df^1}{dx_1}&\cdots&\frac{df^1}{dy_{n-1}}\\\frac{df^2}{dx_1}&\cdots&\frac{df^2}{dy_{n-1}}\end{array}\right)_{2\times(2n-2)}.$$
This map takes the CR singular point $\vec0\in M\subseteq(\re^{2n},J)$
to the complex hyperplane $v^0\in\cc^+\subseteq SG(2n-2,\re^{2n})$,
with coordinates $V=0_{2\times(2n-2)}$.  If we arrange the above two
rows into column vector format (as in the end of Subsection
\ref{lcg}), then the differential of the gauss map at the origin has
matrix representation: \small
$$\left.\left(\begin{array}{ccccc}\frac
d{dx_1}\left(\frac{df^1}{dx_1}\right)&\frac
d{dy_1}\left(\frac{df^1}{dx_1}\right)&\ldots&\frac
d{dx_{n-1}}\left(\frac{df^1}{dx_1}\right)&\frac
d{dy_{n-1}}\left(\frac{df^1}{dx_1}\right)\\\frac
d{dx_1}\left(\frac{df^1}{dy_1}\right)&\frac
d{dy_1}\left(\frac{df^1}{dy_1}\right)&\ldots&\frac
d{dx_{n-1}}\left(\frac{df^1}{dy_1}\right)&\frac
d{dy_{n-1}}\left(\frac{df^1}{dy_1}\right)\\\vdots&&&&\vdots\\\frac
d{dx_1}\left(\frac{df^1}{dx_{n-1}}\right)&\frac
d{dy_1}\left(\frac{df^1}{dx_{n-1}}\right)&\ldots&\frac
d{dx_{n-1}}\left(\frac{df^1}{dx_{n-1}}\right)&\frac
d{dy_{n-1}}\left(\frac{df^1}{dx_{n-1}}\right)\\\frac
d{dx_1}\left(\frac{df^1}{dy_{n-1}}\right)&\frac
d{dy_1}\left(\frac{df^1}{dy_{n-1}}\right)&\ldots&\frac
d{dx_{n-1}}\left(\frac{df^1}{dy_{n-1}}\right)&\frac
d{dy_{n-1}}\left(\frac{df^1}{dy_{n-1}}\right)\\\frac
d{dx_1}\left(\frac{df^2}{dx_1}\right)&\frac
d{dy_1}\left(\frac{df^2}{dx_1}\right)&\ldots&\frac
d{dx_{n-1}}\left(\frac{df^2}{dx_1}\right)&\frac
d{dy_{n-1}}\left(\frac{df^2}{dx_1}\right)\\\frac
d{dx_1}\left(\frac{df^2}{dy_1}\right)&\frac
d{dy_1}\left(\frac{df^2}{dy_1}\right)&\ldots&\frac
d{dx_{n-1}}\left(\frac{df^2}{dy_1}\right)&\frac
d{dy_{n-1}}\left(\frac{df^2}{dy_1}\right)\\\vdots&&&&\vdots\\\frac
d{dx_1}\left(\frac{df^2}{dx_{n-1}}\right)&\frac
d{dy_1}\left(\frac{df^2}{dx_{n-1}}\right)&\ldots&\frac
d{dx_{n-1}}\left(\frac{df^2}{dx_{n-1}}\right)&\frac
d{dy_{n-1}}\left(\frac{df^2}{dx_{n-1}}\right)\\\frac
d{dx_1}\left(\frac{df^2}{dy_{n-1}}\right)&\frac
d{dy_1}\left(\frac{df^2}{dy_{n-1}}\right)&\ldots&\frac
d{dx_{n-1}}\left(\frac{df^2}{dy_{n-1}}\right)&\frac
d{dy_{n-1}}\left(\frac{df^2}{dy_{n-1}}\right)\end{array}\right)\right|_{\mbox{\scriptsize{${\left(\!\!\begin{array}{c}0\\\vdots\\0\end{array}\!\!\right)}$}}}$$
\normalsize
$$=\left(\begin{array}{c}Hf^1\\Hf^2\end{array}\right)_{2(2n-2)\times(2n-2)},$$
  where $Hf^1$ and $Hf^2$ are the real $(2n-2)\times(2n-2)$ real
  Hessian matrices of second derivatives, evaluated at
  $(x_1,\ldots,y_{n-1})^T=(0,\ldots,0)^T$.  The tangent space of the image
  $\bg_s(M)$ at $\bg_s(\vec0)=v^0$ is spanned by the columns of
  this matrix, so in the coordinate chart around $v^0$, it is a
  $(2n-2)$-dimensional subspace that meets $\cc^+$ transversely if the
  columns of this matrix are independent:
\begin{equation}
  \left(\begin{array}{cc}{\mathbbm1}&Hf^1\\J&Hf^2\end{array}\right)_{2(2n-2)\times2(2n-2)}.\label{eq22}
\end{equation}
So, $M$ has a CR singularity in general position if and only if the
determinant of the above matrix is nonzero.  The intersection index
at the origin is related to the sign of the determinant: $ind=+1$ for
$\det>0$, $ind=-1$ for $\det<0$.

This product has the same determinant:
$$\left(\begin{array}{cc}{\mathbbm1}&0\\{\mathbbm1}&J\end{array}\right)\cdot\left(\begin{array}{cc}{\mathbbm1}&Hf^1\\J&Hf^2\end{array}\right)=\left(\begin{array}{cc}{\mathbbm1}&Hf^1\\0&Hf^1+J\cdot
Hf^2\end{array}\right),$$ so calculating the determinant reduces to a
smaller, $(2n-2)\times(2n-2)$ real determinant, $\det(Hf^1+J\cdot
Hf^2)$.

In the $m=4$, $n=3$ case, the $4\times4$ matrix $Hf^1+J\cdot Hf^2$ is:
\begin{equation}\label{eq23}
  \left(\begin{array}{cccc}f_{x_1x_1}^1-f_{y_1x_1}^2&f_{x_1y_1}^1-f_{y_1y_1}^2&f_{x_1x_2}^1-f_{y_1x_2}^2&f_{x_1y_2}^1-f_{y_1y_2}^2\\f_{y_1x_1}^1+f_{x_1x_1}^2&f_{y_1y_1}^1+f_{x_1y_1}^2&f_{y_1x_2}^1+f_{x_1x_2}^2&f_{y_1y_2}^1+f_{x_1y_2}^2\\f_{x_2x_1}^1-f_{y_2x_1}^2&f_{x_2y_1}^1-f_{y_2y_1}^2&f_{x_2x_2}^1-f_{y_2x_2}^2&f_{x_2y_2}^1-f_{y_2y_2}^2\\f_{y_2x_1}^1+f_{x_2x_1}^2&f_{y_2y_1}^1+f_{x_2y_1}^2&f_{y_2x_2}^1+f_{x_2x_2}^2&f_{y_2y_2}^1+f_{x_2y_2}^2\end{array}\right),
\end{equation}
evaluated at $(x_1,y_1,x_2,y_2)^T=(0,0,0,0)^T$.
These real number entries can be expressed in terms of the derivatives
at the origin of the original function $h(z,\z)$:
\footnotesize
$$2\!\left(\!\!\begin{array}{cccc}\rp(h_{z_1\z_1}\!+\!h_{\z_1\z_1})&\ip(-h_{z_1\z_1}\!+\!h_{\z_1\z_1})&\rp(h_{\z_1z_2}\!+\!h_{\z_1\z_2})&\ip(-h_{\z_1z_2}\!+\!h_{\z_1\z_2})\\\ip(h_{z_1\z_1}\!+\!h_{\z_1\z_1})&\rp(h_{z_1\z_1}-h_{\z_1\z_1})&\ip(h_{\z_1z_2}\!+\!h_{\z_1\z_2})&\rp(h_{\z_1z_2}-h_{\z_1\z_2})\\\rp(h_{z_1\z_2}\!+\!h_{\z_1\z_2})&\ip(-h_{z_1\z_2}\!+\!h_{\z_1\z_2})&\rp(h_{\z_2z_2}\!+\!h_{\z_2\z_2})&\ip(-h_{\z_2z_2}\!+\!h_{\z_2\z_2})\\\ip(h_{z_1\z_2}\!+\!h_{\z_1\z_2})&\rp(h_{z_1\z_2}-h_{\z_1\z_2})&\ip(h_{\z_2z_2}\!+\!h_{\z_2\z_2})&\rp(h_{\z_2z_2}-h_{\z_2\z_2})\end{array}\!\!\!\right)\!\!\!,$$
\normalsize evaluated at $z=(0,0)^T$.  We note that these entries do
not depend on the second $z$-derivatives $h_{z_jz_k}$, which are
determined by the coefficient matrix $Q$ in (\ref{eq21}).  This agrees
with the notion that transversality and the index should not depend on
the local holomorphic coordinate system, since it was shown in Section
\ref{sec1} how the coefficients $Q$ could be arbitrarily altered by
holomorphic coordinate changes.  However, it is not as easy to see
from the form of (\ref{eq22}) or (\ref{eq23}) that they do not depend
on the $Q$ coefficients.  We also see in the above matrix that the
entries depend only on the coefficients from $R$ and $S$ in the
expression (\ref{eq21}) for $h(z,\z)$, and not on $e(z,\z)=O(3)$.

  By Lemma \ref{lem6.1} (the proof of which is left to Subsection
  \ref{smc}), the determinant of the above matrix is equal to
  $$2^4\det\left(\begin{array}{cc}R&2S\\\overline{2S}&\overline{R}\end{array}\right).$$
  The above calculations (including the Lemma) generalize to other
  dimensions $n$, and the index formula is even simpler with the
  defining equation in the form (\ref{eq12}), with $(n-1)\times(n-1)$
  complex symmetric coefficient matrix $P=2\bar S$:

\begin{thm}\label{thm5.1}
  Given a real $(2n-2)$-submanifold $M$ in $\co^n$ with a CR singular
  point in standard position and local defining equation:
\begin{eqnarray}
  z_n&=&\z^TRz+\rp\left(z^TPz\right)+e(z,\z),\label{eq46}
\end{eqnarray}
  then $M$ is in general position if and only if the matrix
  $$\Gamma=\left(\begin{array}{cc}R&\overline{P}\\P&\overline{R}\end{array}\right)$$
  is nonsingular.  Further, if $M$ is given an orientation agreeing
  with the orientation of the complex
  $(z_1,\ldots,z_{n-1})$-hyperplane tangent to $M$ at $\vec0$, then
  the intersection index $(\pm1)$ is the sign of the determinant
  $\det(\Gamma)$.  \boxx
\end{thm}
\begin{example}
  In the $n=2$ case, if $M$ is in standard position, oriented to agree
  with the orientation of the $z_1$-axis near the origin, and has
  defining equation
  $$z_2=z_1\z_1+\gamma_1 z_1^2+\bar\gamma_1\z_1^2+e(z_1,\z_1)=z_1\z_1+\rp(2\gamma_1
  z_1^2)+e(z_1,\z_1)\mbox{, $\gamma_1\in\co$},$$ then the index is the sign of
  $\det\left(\begin{array}{cc}1&2\bar\gamma_1\\2\gamma_1&1\end{array}\right)=1-4\gamma_1\bar\gamma_1$,
    which is $+1$ for $0\le|\gamma_1|<\frac12$ and $-1$ for
    $|\gamma_1|>\frac12$.  $M$ is not in general position for
    $|\gamma_1|=\frac12$ (the ``parabolic'' case).
\end{example}

For $n$ in general, the real number value of $\det(\Gamma)$ is not an
invariant under holomorphic transformations, but the sign of the
determinant is; generalizing the action of the group from (\ref{eq13})
to $n$ dimensions,
\begin{equation}\label{eq45}
  (R,P)\mapsto(c\bar A^TRA,\bar cA^TPA),
\end{equation}
for $A=A_{(n-1)\times(n-1)}$, $c=c_{n,n}$.  The determinant of
$\Gamma$ transforms as:
\begin{eqnarray}
  &&\det\left(\begin{array}{cc}c\bar A^TRA&\overline{\bar
      cA^TPA}\\\bar cA^TPA&\overline{c\bar
      A^TRA}\end{array}\right)\nonumber\\ &=&\det\left(\left(\begin{array}{cc}\bar
    A^T&0\\0&A^T\end{array}\right)\left(\begin{array}{cc}cR&c\overline{P}\\\bar
      cP&\bar
      c\overline{R}\end{array}\right)\left(\begin{array}{cc}A&0\\0&\bar
      A\end{array}\right)\right)\label{eq25}\\ &=&|c|^{2(n-1)}|\det
      A|^4\det\left(\begin{array}{cc}R&\overline{P}\\P&\overline{R}\end{array}\right).\nonumber
\end{eqnarray}
Even if the determinant is zero, (\ref{eq25}) shows that the rank of
$\Gamma$ is a biholomorphic invariant.

It also follows from the transformation formula that the vector
$$\left(|\det(R)|^2,|\det(P)|^2,\det(\Gamma)\right)\in\re^3$$ has an
invariant direction (the expression is invariant modulo positive
scalar multiplication), and ratios such as
$\displaystyle{\frac{\det(\Gamma)}{|\det(R)|^2}}$ are numerical
invariants (when well-defined).

\subsection{Invariants at flat points}\label{ifp}

\ 

In the holomorphically flat case (as in Subsections \ref{df},
\ref{hf}), where in some local coordinates $h(z,\z)$ is real valued
and so $M$ is a real $(2n-2)$-hypersurface inside
$\re^{2n-1}=\{y_n=0\}$, and $f^2=0$ in (\ref{eq24}), the
$(2n-2)\times(2n-2)$ matrix $Hf^1+J\cdot Hf^2$ (\ref{eq23}) is just
the Hessian $Hf^1$.  Since $Hf^1$ is a real symmetric matrix, it has
all real eigenvalues, and it follows from the above construction and
the Proof of Lemma \ref{lem6.1} that $\Gamma$ is Hermitian symmetric,
with real eigenvalues equal to $\frac12$ times the eigenvalues of
$Hf^1$.  When $M$ is oriented as in Theorem \ref{thm5.1}, so the
normal vector at the origin is in the positive $x_n$ direction, the
Hessian is exactly the matrix representation of the Weingarten shape
operator at the origin (\cite{thorpe}).  Its determinant is the
Gauss-Kronecker curvature (the product of the $2n-2$ real eigenvalues,
which are the principal curvatures).  So, when $M$ is a hypersurface
in $\re^{2n-1}$, in general position, and positively oriented at a CR
singular point, the index and the curvature have the same sign, which
is \cite{lai} Lemma 4.11.  If we consider holomorphic coordinate
changes that preserve the property that $M$ is contained in
$\re^{2n-1}$ (as in Subsection \ref{hf}), then the matrix $A$ in the
transformations (\ref{eq45}), (\ref{eq25}) may be arbitrary, but if
$R\ne0$, then $c$ must be real.  The $A$ part acts as a Hermitian
congruence transformation on $\Gamma$, preserving its rank
$\rho(\Gamma)$ and signature $(p,q)$, but if $c$ is negative, then the
signature is switched to $(q,p)$.  For $R=0$, the transformation
(\ref{eq25}) is a Hermitian congruence transformation of $\Gamma$ for
any complex $c\ne0$.  We can conclude that for $M\subseteq\re^{2n-1}$,
the rank $\rho(\Gamma)$ and the quantity $\sigma(\Gamma)=|p-q|$ (from
which one can recover the signature, modulo switching, of $\Gamma$)
are invariant under holomorphic transformations preserving the real
valued property of $h(z,\z)$.  Since the action of $(c,A)$ on $\Gamma$
is not affected by the higher degree terms, $\rho(\Gamma)$ and
$\sigma(\Gamma)$ are also invariants of a defining function $h(z,\z)$
in a quadratically flat normal form, under holomorphic transformations
preserving the property of being in a quadratically flat normal form.

For Hermitian $R$, if we think of the quadratic part of (\ref{eq46}),
$$\z^TRz+\rp\left(z^TPz\right),$$ as a real valued quadratic form on
$\re^{2n-2}$, then its zero locus is a real algebraic variety.  The
dimension of this variety, and whether it is reducible or irreducible,
are properties that are invariant under scalar multiplication of the
form, and also under real linear coordinate changes of the domain
$\re^{2n-2}$.  In particular, they will also be invariants under
complex linear transformations $z\mapsto Az$ of $\co^{n-1}$.  Real
valued quadratic forms on $\co^{n-1}$ are considered by \cite{cs},
where the zero set is called a ``quadratic cone'' if it is irreducible
and has dimension $2n-3$.  The action of scalar multiplication and
complex linear transformations $A$ on the set of equations of
quadratic cones in $\co^{n-1}$ is the same as the above action
(\ref{eq45}) on the set of matrix pairs $(R,P)$ appearing in a
quadratically flat normal form.

In the $n=3$ case, a list of equivalence classes of real valued
quadratic forms defining quadratic cones in $\co^2$ is given by
\cite{cs}, and each type of cone corresponds to one of the normal
forms for $2\times2$ matrix pairs $(R,P)$ computed in Subsection
\ref{qnf}.  Our list of normal forms for pairs $(R,P)$, with $R$
Hermitian, (summarized in Examples \ref{ex7.3} -- \ref{ex7.7}) is
longer since \cite{cs} excludes the reducible and lower-dimensional
varieties.

\subsection{Some matrix calculations}\label{smc}

\ 

Let
$R=\left(\begin{array}{cc}\upalpha&\upbeta\\\upgamma&\updelta\end{array}\right)$
and $S=\left(\begin{array}{cc}\rm{a}&\rm{b}\\\rm{c}&\rm{d}\end{array}\right)$ be
$2\times2$ matrices with arbitrary complex entries.  Consider the
following $4\times4$ matrices with real entries:
\begin{eqnarray*}
  R^\prime&=&\left(\begin{array}{cccc}\rp(\upalpha)&-\ip(\upalpha)&\rp(\upbeta)&-\ip(\upbeta)\\\ip(\upalpha)&\rp(\upalpha)&\ip(\upbeta)&\rp(\upbeta)\\\rp(\upgamma)&-\ip(\upgamma)&\rp(\updelta)&-\ip(\updelta)\\\ip(\upgamma)&\rp(\upgamma)&\ip(\updelta)&\rp(\updelta)\end{array}\right)\\
  S^\prime&=&\left(\begin{array}{cccc}\rp(\rm{a})&\ip(\rm{a})&\rp(\rm{b})&\ip(\rm{b})\\\ip(\rm{a})&-\rp(\rm{a})&\ip(\rm{b})&-\rp(\rm{b})\\\rp(\rm{c})&\ip(\rm{c})&\rp(\rm{d})&\ip(\rm{d})\\\ip(\rm{c})&-\rp(\rm{c})&\ip(\rm{d})&-\rp(\rm{d})\end{array}\right).
\end{eqnarray*}

\begin{lem}\label{lem6.1}
  $\det(R^\prime+S^\prime)_{4\times4}=\det\left(\begin{array}{cc}R&S\\\overline{S}&\overline{R}\end{array}\right)_{4\times4}$.
\end{lem}
\begin{proof}
Let
$$K=\frac{\sqrt{2}}2\left(\begin{array}{cccc}1&i&0&0\\0&0&1&i\\1&-i&0&0\\0&0&1&-i\end{array}\right).$$
  Then $K$ is a unitary matrix with $\det(K)=1$.  A calculation shows
  that:
  \begin{eqnarray*}
    KR^\prime\bar K^T&=&\left(\begin{array}{cc}R&0\\0&\overline{R}\end{array}\right),\\
    KS^\prime\bar K^T&=&\left(\begin{array}{cc}0&S\\\overline{S}&0\end{array}\right).
  \end{eqnarray*}
Since $\det(R^\prime+S^\prime)=\det(K(R^\prime+S^\prime)\bar K^T)$,
the claimed formula follows.
\end{proof}
The Lemma generalizes to complex $\mathsf{n}\times \mathsf{n}$
matrices and the corresponding $2\mathsf{n}\times2\mathsf{n}$ matrices
following the same pattern.  In the $\mathsf{n}=1$ case, the analogous
matrix $K$ is
$K=\frac{\sqrt{2}}2\left(\begin{array}{cc}1&i\\1&-i\end{array}\right)$,
  for $\mathsf{n}=3$,
$$K=\frac{\sqrt{2}}2\left(\begin{array}{cccccc}1&i&0&0&0&0\\0&0&1&i&0&0\\0&0&0&0&1&i\\1&-i&0&0&0&0\\0&0&1&-i&0&0\\0&0&0&0&1&-i\end{array}\right),$$
  etc.  Even without the result of the Lemma, it is easy to see that
  elementary identities imply
  $\left(\begin{array}{cc}R&S\\\overline{S}&\overline{R}\end{array}\right)_{2\mathsf{n}\times2\mathsf{n}}$
  has a real determinant.  In the application of the Lemma in
  Subsection \ref{adf}, the $S$ block is assumed to be symmetric ---
  however, the Lemma does not need that assumption.

\section{Complexification}\label{complexification}

We return to the description of $M\subseteq\re^{2n}=\co^{n}$ as the
image of a real analytic parametric map, with domain
$\ddd\subset\re^{2n-2}$ and target $\co^n$.  Re-writing Equation
(\ref{eq24}) in complex form gives:
$$z\mapsto(z,h(z,\z))=(z_1,\ldots,z_{n-1},h(z_1,\z_1,\ldots,z_{n-1},\z_{n-1})).$$
The following complex analytic map, with domain
$\ddd_{\cl}\subseteq\co^{2n-2}$ (as in (\ref{eq5})) and target $\co^n$, is
a complexification of the above parametric map:
$$(z_1,\ldots,z_{n-1},w_1,\ldots,w_{n-1})\mapsto(z_1,\ldots,z_{n-1},h(z_1,w_1,\ldots,z_{n-1},w_{n-1})).$$
The coordinates $w=(w_1, \ldots, w_{n-1})$ are new complex variables;
the new map restricted to $w=\z$ (by substitution in the series
expansion of $h(z,\z)$, as in (\ref{eq-2})) is exactly the original map.  A
geometric interpretation of such a complexification construction (as
the composite of a holomorphic embedding $\ddd_{\cl}\to\co^{2n}$ and a
linear projection $\co^{2n}\to\co^n$) is given in \cite{mem} \S 4; the
$n=2$ case is used in \cite{mw}.  The complex map is singular at the
origin (its complex Jacobian drops rank there since $h(z,w)$ has no linear
terms).

The origin-preserving local biholomorphic transformations of $\co^n$,
$\vec\tz=C\vec z+\vec p(\vec z)$, as in (\ref{eq-1}), act on the
function $h(z,\z)$; this induces an action on the complex map
$(z,w)\mapsto(z,h(z,w))$.  This group of transformations is a subgroup
of a larger transformation group, which acts on the set of (germs at
the origin of) maps $\co^{2n-2}\to\co^n$, by composition with
origin-preserving biholomorphic transformations of both the domain and
the target.  The algebraic interpretation of the complexification
construction is that the larger group allows the transformation of the
$z$ and $w$ (formerly $\z$) variables independently.  Any invariants
under the larger group action will also be invariants under the action
of the subgroup.

In the case $n=2$, the map $(z_1,w_1)\mapsto(z_1,h(z_1,w_1))$ can be
put into one of three normal forms under the larger group:
$(z_1,w_1^2+O(3))$, $(z_1,z_1w_1+O(3))$, $(z_1,O(3))$.  As described
in \cite{mem}, the first two cases correspond, respectively, to
Whitney's fold and cusp singularities of maps $\co^2\to\co^2$.  In the
fold case, a point near the origin in the target $\co^2$ has two
inverse image points.

In the case $n=3$, the map $\ddd_{\cl}\to\co^3$:
$$(z_1,z_2,w_1,w_2)\mapsto(z_1,z_2,h(z_1,w_1,z_2,w_2))$$ can be
transformed to $(\tz_1,\tz_2,\tilde Q(\tz,\tw)+O(3))$, where $\tilde
Q(\tz,\tw)$ is the quadratic part, falling into one of the following
six normal forms (the calculation is omitted):
$$\tw_1^2+\tw_2^2, \ \ \tz_2\tw_2+\tw_1^2, \ \ \tw_1^2,
\ \ \tz_1\tw_1+\tz_2\tw_2, \ \ \tz_1\tw_1, \ \ 0.$$ With $h(z,\z)$ of the
form (\ref{eq0}), the rank of the coefficient matrix $S$ is an
invariant of the complexification under the large group, and so is the
rank of $(R|S)_{2\times4}$.  These two numbers uniquely determine the
equivalence class of the quadratic part under the larger group; the
rank of $R$ is not an invariant.  The first of the above six cases is
the generic one, where $\rho(S)=2$ and the inverse image of a point
near the origin in the target $\co^3$ is a complex analytic curve in
$\ddd_{\cl}\subseteq\co^4$.

\section{Various global examples}\label{vge}

Here we collect some examples of compact real $4$-manifolds embedded
in complex $3$-manifolds.

\begin{example}\label{ex6.1}
  The real $4$-sphere has a real algebraic embedding in a real
  hyperplane $\re^5\subseteq\co^3$, which is (globally)
  holomorphically flat.  For positive coefficients $\mathsf{d}_1$, $\mathsf{d}_2$,
  $\mathsf{d}_3$, $\mathsf{d}_4$, $\mathsf{d}_5$, the implicit equation
  $$\mathsf{d}_1x_1^2+\mathsf{d}_2y_1^2+\mathsf{d}_3x_2^2+\mathsf{d}_4y_2^2+\mathsf{d}_5x_3^2=1$$
  defines an ellipsoidal hypersurface in the real hyperplane $y_3=0$.
  There are exactly two CR singularities, where the tangent space is
  parallel to the $(z_1,z_2)$-subspace.  The two points are
  holomorphically equivalent to each other, and the two local real
  defining equations can be put into a complex normal form
  (\ref{eq12}) with $N={\mathbbm1}$, and $P$ real diagonal.  The
  Hessian at each point is definite, with the entries of $P$ in the
  interval $[0,1)$, depending on $\mathsf{d}_1,\ldots,\mathsf{d}_5$.
    Putting an orientation on $S^4$ induces opposite orientations at
    the two points, so $I_+=I_-=1$, consistent with the characteristic
    class formulas from Subsection \ref{ccf}: for any immersion of
    $S^4$ in general position in $\co^3$, $I_++I_-=\upchi(S^4)=2$, and
    $I_+-I_-=-p_1S^4=0$.
\end{example}

\begin{example}\label{ex6.2b}
  Every compact, oriented, three-dimensional, smooth manifold $M^3$ admits a
  smooth immersion in $\re^4$, $\tau_1:M^3\to\re^4$ (\cite{hirsch},
  \cite{jl}).  Not every such $3$-manifold $M^3$ admits an embedding,
  but the manifolds $S^3$, $S^2\times S^1$, and $S^1\times S^1\times
  S^1$ all can be embedded as hypersurfaces of revolution in $\re^4$.
  There is an immersion (but not an embedding) of $\re P^3$ in $\re^4$
  (\cite{hirsch}, \cite{milnor}).  Also consider any oriented
  immersion of the circle, $\tau_2:S^1\to\re^2$.  For any (almost)
  complex structures on $\re^4$ and $\re^2$, $\tau_1$ is a CR regular
  immersion, $\tau_2$ is a totally real immersion, and the product
  $\tau_1\times\tau_2:M^3\times S^1\to\re^4\times\re^2$ is an oriented,
  CR regular immersion (with respect to the product complex
  structure).  The index sums $I_+=I_-=0$ are consistent with the
  topological formulas from Subsection \ref{ccf}, since $\upchi(M^3\times
  S^1)$ and $p_1(M^3\times S^1)$ are both zero.
\end{example}

\begin{example}\label{ex6.3}
  The $4$-manifold $S^2\times S^2$ does not admit any CR regular
  immersion in $\co^3$; an immersion in general position will have
  $I_++I_-=\upchi(S^2\times S^2)=4$ and $I_+-I_-=-p_1(S^2\times S^2)=0$,
  so $I_+=I_-=2$.  We consider a real algebraic embedding in
  $\re^3\times\re^3$, given by a product of ellipsoids:
  \begin{eqnarray}
    \mathsf{a}x_1^2+\mathsf{b}y_1^2+\mathsf{c}x_3^2&=&1,\label{eq40}\\
    \mathsf{d}x_2^2+\mathsf{e}y_2^2+\mathsf{f}y_3^2&=&1,\nonumber
  \end{eqnarray}
  with positive coefficients $\mathsf{a},\ldots,\mathsf{f}$.  There are exactly four CR
  singularities, at the points with $z_1=z_2=0$.  Solving the real
  defining equations for $x_3$ and $y_3$, setting $z_3=x_3+iy_3$, and
  translating the CR singularities into standard position, the local
  equation at each point is real analytic:
  \begin{eqnarray*}
    x_3+iy_3&=&\frac{\eta_1}{\sqrt{\mathsf{c}}}\left((-1)+\sqrt{1-\mathsf{a}x_1^2-\mathsf{b}y_1^2}\right)\\
            &&\ +i\frac{\eta_2}{\sqrt{\mathsf{f}}}\left((-1)+\sqrt{1-\mathsf{d}x_2^2-\mathsf{e}y_2^2}\right),\\
    z_3&=&\frac{\eta_1}{2\sqrt{\mathsf{c}}}\left(\mathsf{a}\cdot\left(\frac{z_1+\z_1}{2}\right)^2+\mathsf{b}\cdot\left(\frac{z_1-\z_1}{2i}\right)^2\right)\\
    &&\ +i\frac{\eta_2}{2\sqrt{\mathsf{f}}}\left(\mathsf{d}\cdot\left(\frac{z_2+\z_2}{2}\right)^2+\mathsf{e}\cdot\left(\frac{z_2-\z_2}{2i}\right)^2\right)+O(4),\\
  \end{eqnarray*}
  where the four CR singular points correspond to the four sign
  choices $\eta_1=\pm1$, $\eta_2=\pm1$.  After a transformation
  (\ref{eq4}), the equation can be written in the form (\ref{eq12}),
  with diagonal coefficient matrices $(R,P)$:
\begin{eqnarray*}
  z_3&=&(\z_1,\z_2)\left(\begin{array}{cc}\eta_1\frac{\mathsf{a}+\mathsf{b}}{4\sqrt{\mathsf{c}}}&0\\0&i\eta_2\frac{\mathsf{d}+\mathsf{e}}{4\sqrt{\mathsf{f}}}\end{array}\right)\left(\begin{array}{c}z_1\\z_2\end{array}\right)\\
  &&+\rp\left((z_1,z_2)\left(\begin{array}{cc}\eta_1\frac{\mathsf{a}-\mathsf{b}}{4\sqrt{\mathsf{c}}}&0\\0&\eta_2i\frac{\mathsf{d}-\mathsf{e}}{4\sqrt{\mathsf{f}}}\end{array}\right)\left(\begin{array}{c}z_1\\z_2\end{array}\right)\right)+O(4).
\end{eqnarray*}
  The normal form for the coefficient matrix pair, from Case (1a) of
  Subsection \ref{qnf}, has the same form for all four points: for
  $\eta_1=\eta_2$,
  $$(R,P)\sim\left(\left(\begin{array}{cc}1&0\\0&i\end{array}\right),\left(\begin{array}{cc}\frac{|\mathsf{a}-\mathsf{b}|}{\mathsf{a}+\mathsf{b}}&0\\0&\frac{|\mathsf{d}-\mathsf{e}|}{\mathsf{d}+\mathsf{e}}\end{array}\right)\right),$$
  and for $\eta_1=-\eta_2$,
  $$(R,P)\sim\left(\left(\begin{array}{cc}1&0\\0&i\end{array}\right),\left(\begin{array}{cc}\frac{|\mathsf{d}-\mathsf{e}|}{\mathsf{d}+\mathsf{e}}&0\\0&\frac{|\mathsf{a}-\mathsf{b}|}{\mathsf{a}+\mathsf{b}}\end{array}\right)\right),$$
  not depending on $\mathsf{c}$, $\mathsf{f}$.  The non-negative diagonal entries are
  local biholomorphic invariants.  For any values of $\mathsf{a}$, $\mathsf{b}$, $\mathsf{d}$,
  $\mathsf{e}$, $\eta_1$, $\eta_2$, the matrix pair satisfies
  $\det(\Gamma)>0$, so the index is $+1$ for each CR singular point;
  there are two complex points and two anticomplex points.

  We also note that this particular embedding of $S^2\times S^2$ is
  contained in the (Levi nondegenerate) real hypersurface
  $$\mathsf{a}x_1^2+\mathsf{b}y_1^2+\mathsf{c}x_3^2+\mathsf{d}x_2^2+\mathsf{e}y_2^2+\mathsf{f}y_3^2=2,$$ an ellipsoid in
  $\co^3$.  So, the $4$-manifold $S^2\times S^2$ admits some
  topological embedding as a hypersurface in $\re^5$, but the
  extrinsic geometry of this product embedding (\ref{eq40}) is that
  globally, it is not contained in a Levi flat hypersurface, and
  locally, it is not quadratically flat at its CR singular points.
\end{example}

\begin{example}\label{ex6.4}
  Consider $M=\co P^2$ with homogeneous coordinates $[z_0:z_1:z_2]$
  and $\aaa=\co P^3$ with coordinates $[Z_0:Z_1:Z_2:Z_3]$.  For each
  $t\in\re$, let 
  \begin{eqnarray}
    \iota_t:\co P^2&\to&\co P^3:\nonumber\\
    \mbox{$[z_0:z_1:z_2]$}&\mapsto&[z_0\cdot \pp:z_1\cdot \pp:z_2\cdot
    \pp:t\cdot \qq],\label{eq39}
  \end{eqnarray}
  where $\pp$ is the polynomial expression
  $$\pp=\pp(z_0,z_1,z_2)=6z_0\z_0+z_1\z_1+6z_2\z_2$$ and $\qq$ is the polynomial expression
  $$\qq=\qq(z_0,z_1,z_2)=2z_0^2\z_1+2z_0z_1\z_2-z_0z_2\z_1+2z_1z_2\z_0.$$
  Note that for any $[z_0:z_1:z_2]$, $\pp\ne0$, and the first three
  components in the RHS of (\ref{eq39}) have no common zeros.  The
  formula (\ref{eq39}) also has a homogeneity property:
  $$\iota_t([\lambda\cdot z_0:\lambda\cdot z_1:\lambda\cdot
    z_2])=[\lambda^2\bar\lambda\cdot z_0\cdot
    \pp:\lambda^2\bar\lambda\cdot z_1\cdot \pp:\lambda^2\bar\lambda\cdot
    z_2\cdot \pp:\lambda^2\bar\lambda\cdot t\cdot \qq],$$ so $\iota_t$ is
  well-defined.  Observe that for $t=0$,
  $$\iota_0([z_0:z_1:z_2])=[z_0\cdot \pp:z_1\cdot \pp:z_2\cdot \pp:0\cdot
    \qq]=[z_0:z_1:z_2:0]$$ is exactly the embedding $\mathbf c$ from
    Example \ref{ex3.7}.  To show that for each $t$, $\iota_t$ is a
    real analytic embedding, and that the family $\iota_t$ is real
    analytic in $t$ (so that this construction is an isotopy as in
    Example \ref{ex3.7}, and an ``unfolding'' as in \cite{mem}), we
    view $\iota_t$ in local affine coordinate charts.

   The restriction of $\iota_t$ to the $\{z_0=1\}$ neighborhood has
   image contained in the $\{Z_0\ne0\}$ neighborhood of the target
   $\co P^3$, and is given by the formula:
   \begin{eqnarray}
    \mbox{$[1:z_1:z_2]$}&\mapsto&\left[1:z_1:z_2:\frac{t\cdot
        \qq(1,z_1,z_2)}{\pp(1,z_1,z_2)}\right],\nonumber\\ (z_1,z_2)&\mapsto&\left(z_1,z_2,\frac{t\cdot(2\z_1+2z_1\z_2-z_2\z_1+2z_1z_2)}{6+z_1\z_1+6z_2\z_2}\right).\label{eq42}
   \end{eqnarray}
    This is a graph over the $(z_1,z_2)$-hyperplane of a rational (in
    $z$, $\z$) function $F_t(z_1,\z_1,z_2,\z_2)=t\cdot \qq/\pp$ with a
    non-vanishing denominator, so this restriction of $\iota_t$ is a
    real analytic embedding.  The dependence on $t$ is real analytic,
    where $\iota_0$ is just the graph of the constant function $0$.
    For each $t\ne0$, the graph is a smooth real submanifold $M_t$ in
    $\co^3$, which is not a complex submanifold but which inherits its
    orientation from the $(z_1,z_2)$-hyperplane.  Its CR singularities
    can be located by solving the system of two complex equations
   \begin{equation}\label{eq41}
    \frac d{d\z_1}F_t=0,\ \frac d{d\z_2}F_t=0.
   \end{equation}
    The solution set does not depend on $t$ (for $t\ne0$), and it is
    easy to check that each of the following five points
    $(\zeta_1,\zeta_2)$ in the domain is a zero of the
    $\z$-derivatives:
    $$\{(0,2),\ (\sqrt{3},1),\ (-\sqrt{3},1),\ (3i,-1),\ (-3i,-1)\}.$$
    The points
    $(\zeta_1,\zeta_2,F_t(\zeta_1,\bar\zeta_1,\zeta_2,\bar\zeta_2))$
    are the CR singularities of $M_t$, each with a positively oriented
    complex tangent space.  Given $F_t$, calculations with the
    assistance of \cite{maple} found the above five points by solving
    four real equations in four real unknowns, and (omitting the
    details) verified that these five points are the only solutions of
    (\ref{eq41}) in this neighborhood.

  The challenge in constructing this example by making a good choice
  for the above $\pp$ and $\qq$ is to find coefficients which are simple
  and sparse enough so that solving the system is a tractable
  computation with a numerically exact solution set, but not so simple
  that $M_t$ is not in general position and has a CR singularity with
  a degenerate normal form ($ind\ne\pm1$).

    To calculate the index of the CR singular point at
    $(\zeta_1,\zeta_2,F_t)$, we will use Theorem \ref{thm5.1} and find
    the sign of $\det(\Gamma)$ --- this is where we need the exact
    coordinates of the CR singularities.  By the form of (\ref{eq42}),
    it will be enough to check the $t=1$ case, since for $t\ne0$,
    $M_t$ is related to $M_1$ by an invertible complex linear
    transformation of $\co^3$.

    Corresponding to the solution $(0,2)$, translating the CR singular
    point to the origin in $\co^3$ gives a multivariable Taylor
    expansion:
    $$F_1(z_1,\z_1,z_2+2,\overline{z_2+2})=\frac4{15}z_1-\frac1{25}z_1z_2-\frac1{25}z_1\z_2-\frac1{30}z_2\z_1+O(3).$$
    The linear term can be eliminated by a complex linear
    transformation $\tz_3=z_3+c_{31}z_1$, which does not change the
    quadratic or higher degree terms and brings $M_1$ into standard
    position (\ref{eq0}).  The $(R,S)$ coefficient matrix pair
    satisfies $S=0_{2\times2}$ and $\det(R)\ne0$, so
    $\Gamma=\left(\begin{array}{cc}R&0\\0&\bar R\end{array}\right)$
      and $\det(\Gamma)=|\det(R)|^2>0$.  This CR singularity has index
      $+1$.

    Similarly, corresponding to the solution $(\sqrt{3},1)$,
    translating to the origin gives the series expansion:
    \begin{eqnarray*}
      &&F_1(z_1+\sqrt{3},\overline{z_1+\sqrt{3}},z_2+1,\overline{z_2+1})-\frac{\sqrt{3}}{3}\\ &=&\frac15z_1-\frac{\sqrt{3}}{15}z_2-\frac{\sqrt{3}}{75}z_1^2+\frac1{15}z_1z_2+\frac{2\sqrt{3}}{75}z_2^2\\ &&-\frac{8\sqrt{3}}{225}z_1\z_1+\frac4{75}z_1\z_2-\frac4{75}z_2\z_1-\frac{8\sqrt{3}}{75}z_2\z_2+O(3).
    \end{eqnarray*}
    Again in this case, the holomorphic terms are irrelevant, $S=0$,
    and $\det(R)\ne0$, so the CR singularity has index $+1$.  Each of
    the remaining three points, by a similar (but omitted)
    calculation, also has index $+1$, so $M_1$ is in general position.

   It remains to check the points ``at infinity,'' where $z_0=0$ and
   $\iota_t$ restricts to a holomorphic linear embedding
   $$[0:z_1:z_2]\mapsto[0:z_1\cdot \pp:z_2\cdot
     \pp:t\cdot0]=[0:z_1:z_2:0].$$ This restriction is one-to-one and
   misses the image of $F_t$ in the $\{Z_0=1\}$ affine neighborhood,
   which shows that $\iota_t$ is one-to-one for each $t$.

   To look for more CR singularities on the line at infinity, we
   consider the restriction of $\iota_t$ to another affine coordinate
   chart.  The restriction of $\iota_t$ to the $\{z_1=1\}$
   neighborhood has image contained in the $\{Z_1\ne0\}$ neighborhood
   of the target $\co P^3$, and is given by the formula:
   \begin{eqnarray}
    \mbox{$[z_0:1:z_2]$}&\mapsto&\left[z_0:1:z_2:\frac{t\cdot
        \qq(z_0,1,z_2)}{\pp(z_0,1,z_2)}\right],\nonumber\\ (z_0,z_2)&\mapsto&\left(z_0,z_2,\frac{t\cdot(2z_0^2+2z_0\z_2-z_0z_2+2z_2\z_0)}{6z_0\z_0+1+6z_2\z_2}\right).\label{eq43}
   \end{eqnarray}
    This is another graph of a rational function
    $G_t(z_0,\z_0,z_2,\z_2)=t\cdot \qq/\pp$ with a non-vanishing
    denominator, which is real analytic in $z$, $\z$ and $t$.  Since
    we have already found all the CR singularities of the form
    $\iota_t([1:z_1:z_2])$, we can simplify the computational problem
    of finding new CR singularities in the graph of $G_t$ by adding
    the equation $z_0=0$ to get this system:
   \begin{equation}\label{eq44}
    \frac d{d\z_0}G_t=0,\ \frac d{d\z_2}G_t=0,\ z_0=0.
   \end{equation}
    The solution set does not depend on $t$ (for $t\ne0$), and a
    calculation with \cite{maple} shows that $(\zeta_0,\zeta_2)=(0,0)$
    is the only solution of (\ref{eq44}).  The graph of $G_t$ is
    already in standard position; at $t=1$, the quadratic part of the
    defining function is the numerator from (\ref{eq43}):
    $$G_1(z_0,\z_0,z_2,\z_2)=2z_0^2+2z_0\z_2-z_0z_2+2z_2\z_0+O(3).$$
    This is a CR singularity with index $+1$; the graph is in general
    position.

    There is one last point to check, $\iota_t([0:0:1])=[0:0:1:0]$.
    The restriction of $\iota_t$ to the $\{z_2=1\}$ neighborhood is
    given by the formula:
   \begin{eqnarray}
    \mbox{$[z_0:z_1:1]$}&\mapsto&\left[z_0:z_1:1:\frac{t\cdot
        \qq(z_0,z_1,1)}{\pp(z_0,z_1,1)}\right],\nonumber\\ (z_0,z_1)&\mapsto&\left(z_0,z_1,\frac{t\cdot(2z_0^2\z_1+2z_0z_1-z_0\z_1+2z_1\z_0)}{6z_0\z_0+z_1\z_1+6}\right).\nonumber
   \end{eqnarray}
   Since this image is also a real analytic graph, $\iota_t$ is a
   (global) real analytic embedding depending real analytically on
   $t$.  The origin in this neighborhood is another CR singularity of
   the image; the graph is already in standard position and, for
   $t\ne0$, in general position, with a CR singular point of index
   $+1$.

   The conclusion is that for $t\ne0$, $\iota_t:\co P^2\to\co P^3$ has
   exactly seven CR singular points, at the image
   of $$\{[1:0:2],\ [1:\pm\sqrt{3}:1],\ [1:\pm3i:-1],\ [0:1:0],\ [0:0:1]\}.$$
   Every image point is in $N_2^+$ (positively oriented tangent
   space), with index $+1$, consistent with the characteristic class
   calculations of Example \ref{ex3.7}.
\end{example}

\section{Summary}\label{sum}

\begin{thm}\label{thm4.1}
  Given a real analytic 4-dimensional submanifold $M$ in $\co^3$, for
  any CR singular point there is a local holomorphic coordinate
  neighborhood so that the CR singular point is at the origin, the
  tangent space is the $(z_1,z_2)$-hyperplane, and the local defining
  equation for $M$ is given by
\begin{eqnarray}
  z_3&=&h(z,\z)\nonumber\\ &=&(\z_1,\z_2)N\left(\begin{array}{c}z_1\\z_2\end{array}\right)+\rp\left((z_1,z_2)P\left(\begin{array}{c}z_1\\z_2\end{array}\right)\right)+e(z_1,\z_1,z_2,\z_2),\label{eq38}
\end{eqnarray}
  where $e(z,\z)=O(3)$ is real analytic.  The coefficient matrices $N$, $P$
  fall into one of the cases from Table {\mbox{\rm $1$}}, and exactly
  one (modulo equivalences as indicated).  \boxx
\end{thm}
  In the following table, the third column is the number of real
  moduli for each case.  The last column indicates the sign of
  $\det(\Gamma)$ that can occur for various values of the entries of
  $N$ and $P$: positive, negative, or zero ($+,-,0$).

\pagebreak

\begin{center}
  Table 1
  \begin{tabular}{|c|l|} \hline 
    $N$&\ \ \ \ \ \ $P$\\ \hline
    $\begin{array}{c}\left(\begin{array}{cc}1&0\\0&e^{i\theta}\end{array}\right)\\0<\theta<\pi\end{array}$&$\begin{array}{c|c|lr}\left(\begin{array}{cc}a&b\\b&d\end{array}\right)&5&a>0,d>0,b\sim-b\in\co&+-0\\\left(\begin{array}{cc}0&b\\b&d\end{array}\right)&3&b\ge0,d\ge0&+-0\\\left(\begin{array}{cc}a&b\\b&0\end{array}\right)&3&a>0,b\ge0&\ \ \ \ \ \ \ \ \ +-0\end{array}$\\
    \hline
    $\left(\begin{array}{cc}1&0\\0&1\end{array}\right)$&$\begin{array}{c|c|lr}\left(\begin{array}{cc}a&0\\0&d\end{array}\right)&2&0\le
    a\le d&\ \ \ \ \ \ \ \ \ \ \ \ \ \ \ \ \ \ \ \ \ \ \ \ \ \ \ \ +-0\end{array}$\\ \hline
    $\left(\begin{array}{cc}1&0\\0&-1\end{array}\right)$&$\begin{array}{c|c|lr}\left(\begin{array}{cc}a&0\\0&d\end{array}\right)&2&0\le
    a\le
    d&+-0\\\left(\begin{array}{cc}0&b\\b&0\end{array}\right)&1&b>0&\ \ \ \ \ \ \ \ \ \ \ \ \ \ \ \ \ \ \ \ \ \ \ \ \ \ \ \ \ \ \ \  \ + \\\left(\begin{array}{cc}1&1\\1&1\end{array}\right)&0&&\ \ \ \ \ \ \ \ \ \ \ \ \ \ \ \ \ \ \ \ \ \ \ \ \ \ \ \ \ \ + \end{array}$\\
    \hline
    $\left(\begin{array}{cc}0&1\\1&0\end{array}\right)$&$\begin{array}{c|c|lr}\left(\begin{array}{cc}0&b\\b&1\end{array}\right)&1&b>0&+ \ 0\\\left(\begin{array}{cc}1&0\\0&d\end{array}\right)&2&\ip(d)>0&\ \ \ \ \ \ \ \ \ \ \ \ \ \ \ \ \ \ \ \ \ \ \ \ \ \ \ \ \ \ \ \ \  +\end{array}$\\
    \hline
    $\begin{array}{c}\left(\begin{array}{cc}0&1\\\tau&0\end{array}\right)\\0<\tau<1\end{array}$&$\begin{array}{c|c|lr}\left(\begin{array}{cc}a&b\\b&d\end{array}\right)&5&b>0,|a|=1,(a,d)\sim(-a,-d)\!\!\!\!\!&\ \ \ +-0\\\left(\begin{array}{cc}0&b\\b&d\end{array}\right)&3&b>0,|d|=1,d\sim-d&+-0\\\left(\begin{array}{cc}0&b\\b&0\end{array}\right)&2&b>0&+-0\\\left(\begin{array}{cc}1&0\\0&d\end{array}\right)&3&d\in\co&+ \ 0\\\left(\begin{array}{cc}0&0\\0&1\end{array}\right)&1&&+\\\left(\begin{array}{cc}0&0\\0&0\end{array}\right)&1&&+\end{array}$\\
    \hline
    $\left(\begin{array}{cc}0&1\\0&0\end{array}\right)$&$\begin{array}{c|c|lr}\left(\begin{array}{cc}a&b\\b&1\end{array}\right)&3&b>0,a\in\co&+-0\\\left(\begin{array}{cc}1&b\\b&0\end{array}\right)&1&b>0&+-0\\\left(\begin{array}{cc}0&b\\b&0\end{array}\right)&1&b>0&+-0\\\left(\begin{array}{cc}a&0\\0&1\end{array}\right)&1&a\ge0&+ \ 0\\\left(\begin{array}{cc}1&0\\0&0\end{array}\right)&0&&0\\\left(\begin{array}{cc}0&0\\0&0\end{array}\right)&0&&\ \ \ \ \ \ \ \ \ \ \ \ \ \ \ \ \ \ \ \ \ \ \ \ \ \ \ \ \ \ \ \ \!0\end{array}$\\
    \hline
  \end{tabular}
 
\pagebreak

  Table 1, continued

  \begin{tabular}{|c|l|} \hline 
    $N$&\ \ \ \ \ \ $P$\\ \hline
    $\left(\begin{array}{cc}0&1\\1&i\end{array}\right)$&$\begin{array}{c|c|lr}\left(\begin{array}{cc}a&0\\0&d\end{array}\right)&3&a>0,d\in\co&+-0\\\left(\begin{array}{cc}0&b\\b&0\end{array}\right)&1&b>0&+ \ 0\\\left(\begin{array}{cc}0&0\\0&d\end{array}\right)&1&d\ge0&+\end{array}$\\
    \hline
    $\left(\begin{array}{cc}1&0\\0&0\end{array}\right)$&$\begin{array}{c|c|lr}\left(\begin{array}{cc}a&0\\0&1\end{array}\right)&1&a\ge0&+-0\\\left(\begin{array}{cc}0&1\\1&0\end{array}\right)&0&&+\\\left(\begin{array}{cc}a&0\\0&0\end{array}\right)&1&a\ge0&\ \ \ \ \ \ \ \ \ \ \ \ \ \ \ 0\end{array}$\\
    \hline
    $\left(\begin{array}{cc}0&0\\0&0\end{array}\right)$&$\begin{array}{c|c|lr}\left(\begin{array}{cc}1&0\\0&1\end{array}\right)&0&&+\\\left(\begin{array}{cc}1&0\\0&0\end{array}\right)&0&&0\\\left(\begin{array}{cc}0&0\\0&0\end{array}\right)&0&&\ \ \ \ \ \ \ \ \ \ \ \ \ \ \ \ \ \ \ \ \ \  0\end{array}$\\
    \hline
  \end{tabular}.
\end{center}

\begin{thm}\label{thm7.2}
  Given a real analytic 4-dimensional submanifold $M$ in $\co^3$ with
  local defining equation \mbox{\rm (\ref{eq38})} in one of the normal
  forms from Theorem \mbox{\rm \ref{thm4.1}}, if the quadratic part of
  $h(z,\z)$ in \mbox{\rm (\ref{eq38})} is real valued, then the coefficient
  matrices $N$, $P$ fall into exactly one of the cases from the
  following Examples \mbox{\rm \ref{ex7.3}} -- \mbox{\rm \ref{ex7.7}}.
  \boxx
\end{thm}

The whole numbers in the middle columns are invariants of the defining
function $h(z,\z)$ under the group of local biholomorphic coordinate changes
that preserves the property of being in a quadratically flat normal
form.

As previously mentioned, related lists of normal forms for pairs have
appeared in \cite{iz}, \cite{e}, and \cite{cs}.

\pagebreak

\begin{example}\label{ex7.3}
  For $N=\left(\begin{array}{cc}1&0\\0&1\end{array}\right)$,
    $\rho(N)=2$, $\sigma(N)=2$,
\begin{center}
  \begin{tabular}{|c|c|c|c|c|c|c|}
    \hline
    $P$&&$\rho(P)$&$\rho(N|P)$&$\rho(\Gamma)$&$\sigma(\Gamma)$&sign$(\det(\Gamma))$ \\ \hline 
    $\left(\begin{array}{cc}0&0\\0&0\end{array}\right)$&&$0$&$2$&$4$&$4$&$+$ \\ \hline
    $\left(\begin{array}{cc}0&0\\0&d\end{array}\right)$&$0<d<1$&$1$&$2$&$4$&$4$&$+$ \\ \hline
    $\left(\begin{array}{cc}0&0\\0&1\end{array}\right)$&&$1$&$2$&$3$&$3$&$0$ \\ \hline
    $\left(\begin{array}{cc}0&0\\0&d\end{array}\right)$&$1<d$&$1$&$2$&$4$&$2$&$-$ \\ \hline
    $\left(\begin{array}{cc}a&0\\0&d\end{array}\right)$&$0<a\le d<1$&$2$&$2$&$4$&$4$&$+$ \\ \hline
    $\left(\begin{array}{cc}a&0\\0&1\end{array}\right)$&$0<a<1$&$2$&$2$&$3$&$3$&$0$ \\ \hline
    $\left(\begin{array}{cc}a&0\\0&d\end{array}\right)$&$0<a<1<d$&$2$&$2$&$4$&$2$&$-$ \\ \hline
    $\left(\begin{array}{cc}1&0\\0&1\end{array}\right)$&&$2$&$2$&$2$&$2$&$0$ \\ \hline
    $\left(\begin{array}{cc}1&0\\0&d\end{array}\right)$&$1<d$&$2$&$2$&$3$&$1$&$0$ \\ \hline
    $\left(\begin{array}{cc}a&0\\0&d\end{array}\right)$&$1<a\le d$&$2$&$2$&$4$&0&$+$ \\ \hline

  \end{tabular}.
\end{center}
The diagonal elements of $P$ in the above table are the previously
mentioned ``generalized Bishop invariants.''  This is the only Example
where cases with $\sigma(\Gamma)=4$ occur, i.e., where $\Gamma$ is
definite, or, equivalently, where the real Hessian is definite, as
discussed in Subsection \ref{adf}.  These are the points called ``flat
elliptic'' points by \cite{dtz} and \cite{d}, and they appeared in the
flat embedding of the ellipsoid in Example \ref{ex6.1}.  The
definiteness of $\Gamma$ characterizes this elliptic property among
all the equivalence classes from Theorem \ref{thm7.2} and Examples
\ref{ex7.3} -- \ref{ex7.7}.  The first line in the above chart, where
$P=0$, represents the case considered by \cite{hy}.

We also see a difference from the real surface $M\subseteq\co^2$ case,
where a complex point (in $N_1^+$) has the elliptic property if and
only if $ind=+1$.  That characterization does not generalize to
dimensions $m=4$, $n=3$; even in the above table, flat elliptic
complex points have $ind=+1$, but so do some quadratically flat,
non-elliptic points.

The normal forms in the above table with full rank $\rho(\Gamma)=4$
and indefinite signature $\sigma(\Gamma)=2$ or $0$ represent the
equivalence classes of points called ``hyperbolic'' by \cite{d}.
\end{example}

\pagebreak

\begin{example}\label{ex7.4}
  For $N=\left(\begin{array}{cc}1&0\\0&-1\end{array}\right)$,
    $\rho(N)=2$, $\sigma(N)=0$,
\begin{center}
  \begin{tabular}{|c|c|c|c|c|c|c|}
    \hline
    $P$&&$\rho(P)$&$\rho(N|P)$&$\rho(\Gamma)$&$\sigma(\Gamma)$&sign$(\det(\Gamma))$ \\ \hline 
    $\left(\begin{array}{cc}0&0\\0&0\end{array}\right)$&&$0$&$2$&$4$&$0$&$+$ \\ \hline
    $\left(\begin{array}{cc}0&0\\0&d\end{array}\right)$&$0<d<1$&$1$&$2$&$4$&$0$&$+$ \\ \hline
    $\left(\begin{array}{cc}0&0\\0&1\end{array}\right)$&&$1$&$2$&$3$&$1$&$0$ \\ \hline
    $\left(\begin{array}{cc}0&0\\0&d\end{array}\right)$&$1<d$&$1$&$2$&$4$&$2$&$-$ \\ \hline
    $\left(\begin{array}{cc}a&0\\0&d\end{array}\right)$&$0<a\le d<1$&$2$&$2$&$4$&$0$&$+$ \\ \hline
    $\left(\begin{array}{cc}a&0\\0&1\end{array}\right)$&$0<a<1$&$2$&$2$&$3$&$1$&$0$ \\ \hline
    $\left(\begin{array}{cc}a&0\\0&d\end{array}\right)$&$0<a<1<d$&$2$&$2$&$4$&$2$&$-$ \\ \hline
    $\left(\begin{array}{cc}1&0\\0&1\end{array}\right)$&&$2$&$2$&$2$&$0$&$0$ \\ \hline
    $\left(\begin{array}{cc}1&0\\0&d\end{array}\right)$&$1<d$&$2$&$2$&$3$&$1$&$0$ \\ \hline
    $\left(\begin{array}{cc}a&0\\0&d\end{array}\right)$&$1<a\le d$&$2$&$2$&$4$&0&$+$ \\ \hline
    $\left(\begin{array}{cc}1&1\\1&1\end{array}\right)$&&$1$&$2$&$4$&0&$+$ \\ \hline
    $\left(\begin{array}{cc}0&b\\b&0\end{array}\right)$&$0<b$&$2$&$2$&$4$&0&$+$ \\ \hline

  \end{tabular}.
\end{center}
In the above table, we see that the discrete invariants in the last
$5$ columns are repeated in a few cases, so they are not enough to
distinguish inequivalent matrix normal forms of different shapes (such
as diagonalizable or not).
\end{example}

\begin{example}\label{ex7.5}
  For $N=\left(\begin{array}{cc}0&1\\1&0\end{array}\right)$,
    $\rho(N)=2$, $\sigma(N)=0$,
\begin{center}
  \begin{tabular}{|c|c|c|c|c|c|c|}
    \hline
    $P$&&$\rho(P)$&$\rho(N|P)$&$\rho(\Gamma)$&$\sigma(\Gamma)$&sign$(\det(\Gamma))$ \\ \hline 
    $\left(\begin{array}{cc}0&b\\b&1\end{array}\right)$&$0<b\ne1$&$2$&$2$&$4$&$0$&$+$ \\ \hline
    $\left(\begin{array}{cc}0&1\\1&1\end{array}\right)$&&$2$&$2$&$3$&$1$&$0$ \\ \hline
    $\left(\begin{array}{cc}1&0\\0&d\end{array}\right)$&$0<\ip(d)$&$2$&$2$&$4$&$0$&$+$ \\ \hline
    \end{tabular}.
\end{center}
  Some of the rows in this table have the same $\rho$ and $\sigma$
  data as rows from the previous Example.
\end{example}

\pagebreak

\begin{example}\label{ex7.6}
  For $N=\left(\begin{array}{cc}1&0\\0&0\end{array}\right)$,
    $\rho(N)=1$, $\sigma(N)=1$,
\begin{center}
  \begin{tabular}{|c|c|c|c|c|c|c|}
    \hline
    $P$&&$\rho(P)$&$\rho(N|P)$&$\rho(\Gamma)$&$\sigma(\Gamma)$&sign$(\det(\Gamma))$ \\ \hline 
    $\left(\begin{array}{cc}0&0\\0&1\end{array}\right)$&&$1$&$2$&$4$&$2$&$-$ \\ \hline
    $\left(\begin{array}{cc}a&0\\0&1\end{array}\right)$&$0<a<1$&$2$&$2$&$4$&$2$&$-$ \\ \hline
    $\left(\begin{array}{cc}1&0\\0&1\end{array}\right)$&&$2$&$2$&$3$&$1$&$0$ \\ \hline
    $\left(\begin{array}{cc}a&0\\0&1\end{array}\right)$&$1<a$&$2$&$2$&$4$&$0$&$+$ \\ \hline
    $\left(\begin{array}{cc}0&1\\1&0\end{array}\right)$&&$2$&$2$&$4$&$0$&$+$ \\ \hline
    $\left(\begin{array}{cc}0&0\\0&0\end{array}\right)$&&$0$&$1$&$2$&$2$&$0$ \\ \hline
    $\left(\begin{array}{cc}a&0\\0&0\end{array}\right)$&$0<a<1$&$1$&$1$&$2$&$2$&$0$ \\ \hline
    $\left(\begin{array}{cc}1&0\\0&0\end{array}\right)$&&$1$&$1$&$1$&$1$&$0$ \\ \hline
    $\left(\begin{array}{cc}a&0\\0&0\end{array}\right)$&$1<a$&$1$&$1$&$2$&$0$&$0$ \\ \hline
  \end{tabular}.
\end{center}
\end{example}

\begin{example}\label{ex7.7}
  For $N=\left(\begin{array}{cc}0&0\\0&0\end{array}\right)$,
    $\rho(N)=0$, $\sigma(N)=0$,
\begin{center}
  \begin{tabular}{|c|c|c|c|c|c|c|}
    \hline
    $P$&&$\rho(P)$&$\rho(N|P)$&$\rho(\Gamma)$&$\sigma(\Gamma)$&sign$(\det(\Gamma))$ \\ \hline 
    $\left(\begin{array}{cc}1&0\\0&1\end{array}\right)$&&$2$&$2$&$4$&$0$&$+$ \\ \hline
    $\left(\begin{array}{cc}1&0\\0&0\end{array}\right)$&&$1$&$1$&$2$&$0$&$0$ \\ \hline
    $\left(\begin{array}{cc}0&0\\0&0\end{array}\right)$&&$0$&$0$&$0$&$0$&$0$ \\ \hline
    \end{tabular}.
\end{center}
\end{example}

\end{document}